\documentclass[11pt, reqno]{amsart}
\usepackage{latexsym,amsmath,amssymb, amsthm, mathtools}
\usepackage{cases, verbatim, esint, todonotes}
\usepackage[active]{srcltx}
\usepackage{hyperref}
\usepackage{xcolor}
\usepackage{float}

\usepackage[T1]{fontenc}
\usepackage{mathscinet}

\usepackage[margin=1.5in]{geometry}

\newtheorem{thm}{Theorem}
\newtheorem{lem}[thm]{Lemma}
\newtheorem{prop}[thm]{Proposition}

\theoremstyle{remark}
\newtheorem{rmk}[thm]{Remark}
\newtheorem{example}[thm]{Example}

\theoremstyle{definition}
\newtheorem{defi}[thm]{Definition}

\numberwithin{thm}{section} 
\numberwithin{equation}{section}

\makeatletter

\newcommand{\Rmnum}[1]{\expandafter\@slowromancap\romannumeral #1@}
\makeatother

\makeatletter
\@namedef{subjclassname@2020}{\textup{2020} Mathematics Subject Classification}
\makeatother

\def\R{{\mathbb R}}
\def\X{{\mathbf X}}

\def\M{{\mathcal M}}
\def\N{{\mathcal N}}

\def\S{{\mathcal S}}

\def\E{{\mathcal E}}
\def\O{{\mathcal O}}

\def\V{{\mathcal V}}

\newcommand{\pO}{\partial\Omega}
\newcommand{\Oba}{\overline{\Omega}}

\newcommand{\vep}{\varepsilon}
\newcommand{\ol}{\overline}
\newcommand{\ul}{\underline}
\newcommand{\dive}{\operatorname{div}}

\newcommand{\bpm}{\begin{pmatrix}}
\newcommand{\epm}{\end{pmatrix}}
\newcommand{\la}{\langle}
\newcommand{\ra}{\rangle}

\newcommand{\beq}{\begin{equation}}
\newcommand{\eeq}{\end{equation}}

plus 6pt minus 12pt
\title[Infinity-eigenvalue in metric spaces]{\protect{Principal eigenvalue problem\\ 
for infinity Laplacian in metric spaces}}

\author{Qing Liu}
\address[Qing Liu]{Geometric Partial Differential Equations Unit, Okinawa Institute of Science and Technology 
Graduate University, Japan, {\tt qing.liu@oist.jp}}

\author{Ayato Mitsuishi}
\address[Ayato Mitsuishi]{Department of Applied Mathematics, Faculty of Science, Fukuoka University, Japan, {\tt mitsuishi@fukuoka-u.ac.jp}}

\date{\today}

\begin{document}

\begin{abstract}
This paper is concerned with the Dirichlet eigenvalue problem associated to the $\infty$-Laplacian in metric spaces. We establish a direct PDE approach to find the principal eigenvalue and eigenfunctions in a proper geodesic space without assuming any measure structure. We provide an appropriate notion of solutions to the $\infty$-eigenvalue problem and show the existence of solutions by adapting Perron's method. Our method is different from the standard limit process via the variational eigenvalue formulation for $p$-Laplacian in the Euclidean space. 
\end{abstract}


\maketitle

\section{Introduction}
\subsection{Background and motivation}

In this paper, we consider the principal eigenvalue and eigenfunctions associated to the $\infty$-Laplacian with homogeneous Dirichlet boundary condition in metric spaces. One of our major contributions is a general framework that can be applied to study this eigenvalue problem in a large variety of metric spaces. Throughout this paper, the metric space $(\X, d)$ is assumed to satisfy the following two conditions:
\begin{itemize}
\item $(\X, d)$ is a geodesic space, namely, for any $x, y\in \X$, there exists a Lipschitz curve $\gamma: [a, b]\to \X$ such that $\gamma(a)=x, \gamma(b)=y$ and 
$d(x, y)=\ell(\gamma)$,
where $\ell(\gamma)$ stands for the length of $\gamma$. 

\item $(\X, d)$ is proper, that is, for any $x\in \X$ and $r>0$, the closed metric ball $\ol{B_r(x)}$ is compact. 
Here and in the sequel, we denote by $B_r(x)$ the open metric ball centered at $x$ with radius $r>0$. 
\end{itemize}

Before stating our main results, let us first go over the background on the topic and describe our motivation of this work. The study on the eigenvalue problem for the $\infty$-Laplacian is initiated by the work \cite{JLM1} (and also \cite{FIN}), where the limits of eigenvalue and eigenfunctions for the $p$-Laplacian as $p\to \infty$ are investigated in the Euclidean space. More precisely, for any given $1<p<\infty$ and a bounded domain $\Omega\subset \R^n$, via the Rayleigh quotient we can obtain the first $p$-eigenvalue, written as ${\Lambda_p}^p$ (the $p$-th power of $\Lambda_p$), by
\beq\label{p-rayleigh}
{\Lambda_p}^p=\min \left\{\frac{\int_\Omega |\nabla u|^p \, dx}{\int_\Omega |u|^p \, dx}: u\in W_0^{1, p}(\Omega)\setminus \{0\}\right\}.
\eeq
The minimizers for \eqref{p-rayleigh}, which solve $-\dive(|\nabla u|^{p-2}\nabla u)={\Lambda_p}^p|u|^{p-2}u$  in $\Omega$ with  $u=0$ on $\partial \Omega$, are called $p$-eigenfunctions or $p$-ground states.
It is shown \cite{JLM1, FIN} that $\Lambda_p$ 
converges as $p\to \infty$ to 
\beq\label{inf-eigenvalue}
\Lambda_\infty={1\over R_{\infty}} ,
\eeq
where $R_\infty>0$ is the radius of the maximum ball inscribed in $\Omega$:
\[
R_\infty=\max_{x\in \Oba} \min_{y\in \pO} |x-y|.
\]
Thus, the value $\Lambda_\infty$ is considered as the principal eigenvalue of the $\infty$-Laplacian and is called $\infty$-eigenvalue.  

Moreover, it is also proved in \cite{JLM1, FIN} that a subsequence of normalized $p$-eigenfunctions converges uniformly, as $p\to \infty$,  to  a positive viscosity solution $u$ of the following obstacle problem
\begin{numcases}{}
\min\{|\nabla u|-\Lambda_\infty u, \ -\Delta_\infty u\}=0&\quad \text{in $\Omega$,}\label{inf-eigen}  \\
u= 0&\quad \text{on $\pO$,} \label{dirichlet}
\end{numcases}
where $\Delta_\infty u$ denotes the $\infty$-Laplacian of $u\in C^2(\Omega)$, i.e., 
$\Delta_\infty u=\la \nabla^2 u\nabla u, \nabla u\ra$.
Such positive solutions obtained from the limit process are called $\infty$-ground states in the literature.  The higher eigenvalues and eigenfunctions arising in the limit are studied in \cite{JuLi}. An interesting observation about \eqref{inf-eigen} is that the eigenvalue $\Lambda_\infty$ appears in the first-order component rather than the $\infty$-Laplacian part of the minimum. Throughout this paper we call $|\nabla u|-\lambda u$ with $\lambda>0$ the eikonal operator on $u$. One can see at least formally that, under the transformation $U=\log u$, $|\nabla u|=\lambda u$ can be expressed as the standard eikonal equation $|\nabla U|=\lambda$.

The above results on the principal $\infty$-eigenvalue problem are later developed for more general nonlinear elliptic operators \cite{BeKa, ChPe, PeRo, ChPa, LiLi1, BoRoSa, DaRoSa} and for more general boundary conditions \cite{GaMaPeRo2, Pat, EsKaNiTo}. However, on the other hand, less is known about this problem in general geometric settings. We refer to \cite{BKJ} for generalization under the Finsler metrics. In metric measure spaces, the convergence of $\Lambda_p$ to $\Lambda_\infty$ is addressed in \cite{AHo, Hon} and results related to the limit for higher $p$-eigenvalues are recently provided by the second author \cite{Mit}. In these results the equipped measure structure plays a fundamental role to allow the variational approach. 

It is worth pointing out that the full convergence of normalized $p$-eigenfunctions as $p\to \infty$ is still unclear in general even in the Euclidean space. This is related to the long standing open question on the uniqueness of $\infty$-ground states up to a multiplicative factor, or in other words, the simplicity of the $\infty$-eigenvalue $\Lambda_\infty$. An affirmative result is given by Yu \cite{Yu1} for $\Omega\subset \R^n$ in certain particular shapes when the distance to $\pO$ is a viscosity solution of \eqref{inf-eigen} and \eqref{dirichlet}. Rather than investigating the convergence of $p$-eigenfunctions, one may try to directly prove the uniqueness of viscosity solutions of \eqref{inf-eigen} and \eqref{dirichlet} up to a constant multiple.
However, a counterexample has been constructed in \cite{HSY} to show that it fails in general. 
In addition to \cite{Yu1, HSY}, further progress has been made toward this uniqueness problem; see, for instance, \cite{CrFr5, CrFr4, DaRoSa, LiLi2, LiLi3, LiLi4}. A characterization of the $\infty$-eigenvalue problem is provided by \cite{ChDeJi, BuK} based on the optimal transport theory. The connection between the infinity Laplacian and mass transfer problems through convex duality is addressed in earlier papers such as \cite{EG, BBD, GaMaPeRo1}. 

In this work, from a more geometric perspective,  we look into the $\infty$-eigenvalue problem in general metric spaces with minimal structure assumptions. Our approach actually applies to an arbitrary proper geodesic space without any measure structure, which constitutes a major difference from the known results.

Our study is motivated by the following observations. First, the $\infty$-eigenvalue $\Lambda_\infty$ in \eqref{inf-eigenvalue} is a completely geometric quantity and requires nothing more than the space metric. We naturally expect that in a general metric space $(\X, d)$, the expression of the $\infty$-eigenvalue turns into
\beq\label{inf-eigenvalue general}
\Lambda_\infty=\frac{1}{\max_{x\in \Oba} d(x, \partial \Omega)}.
\eeq
The denominator represents the radius of the maximum inscribed metric ball, which we still denote by $R_\infty$, i.e., 
\beq\label{inrad}
R_\infty=\max_{x\in \Oba} d(x, \partial \Omega).
\eeq
Here and in the sequel, $d(x, E)$ denotes the distance from a point $x\in \X$ to a compact set $E\subset \X$, namely, $d(x, E)=\min_{y\in E} d(x, y)$.

Second, both the eikonal operator and the $\infty$-Laplacian appearing in the nonlinear obstacle problem \eqref{inf-eigen} can be understood under merely the length structure. 
In recent years, several notions of solutions to the eikonal equation in geodesic or length spaces are proposed \cite{GHN, GaS, LShZ}. Concerning the $\infty$-Laplace equation, we refer to \cite{PSSW} for a tug-of-war game interpretation in length spaces. It is also well known that $\infty$-harmonic functions in the Euclidean space can be characterized by comparison with cones \cite{CEG, ArCrJu, AS2}. This characterization is extended to sub-Riemannian manifolds \cite{BiDrMa, Bi4} and general metric spaces \cite{Ju, JuSh}.  None of these results essentially require measures on the spaces. 

Besides, the Dirichlet eigenvalue problem can be set up without relying on the limit process via $p$-Laplacian. Recall that in the Euclidean space the principal eigenvalues for linear elliptic operators are found \cite{BNV} without using the Rayleigh quotient but the maximum principle; see \cite{BiDe1, BiDe2} for further results on the principal eigenvalue problem for fully nonlinear equations.
The eigenvalue is characterized as the maximum value of $\lambda\in \R$ that admits existence of positive viscosity supersolutions.

Based on the observations above, we establish a new approach to the eigenvalue problem that is applicable to general geodesic spaces. Our strategy is as follows. 
Instead of passing to the limit for the $p$-eigenvalue problem as $p\to \infty$ as in \cite{JLM1}, we follow \cite{BiDe1, BiDe2} to investigate, in a more straightforward manner, the maximum value $\lambda$ that guarantees existence of positive supersolutions of
\beq\label{inf-eigen-lambda}
\min\{|\nabla u|-\lambda u, \ -\Delta_\infty u\}=0\quad \text{in $\Omega$.}
\eeq
We adopt this method to avoid the use of measures that are required to formulate the variational $p$-eigenvalue problem. It turns out that such a critical value $\lambda>0$ does coincide with $\Lambda_\infty$ as in \eqref{inf-eigenvalue general}. 
 
Once the eigenvalue $\Lambda_\infty$ is justified, we can discuss positive solutions of \eqref{inf-eigen} that also satisfy the boundary condition \eqref{dirichlet}. Such solutions are regarded as $\infty$-eigenfunctions in our general setting.   
Adapting Perron's method,  we construct solutions of \eqref{inf-eigen} by taking the infimum of all supersolutions satisfying appropriate conditions that essentially play the role of normalization. 
In order to make our PDE-based arguments above work, it is crucial to find an appropriate notion of solutions of \eqref{inf-eigen-lambda}, which we will clarify in a moment. 

 As in the Euclidean case, we do not know whether the $\infty$-eigenfunctions are unique up to a constant multiple in general. 
The setting becomes simpler when $\X$ is a finite metric graph, thanks to the finiteness and one-dimensional structure of the space. 
In our forthcoming work \cite{LMit2}, we study further properties of $\infty$-eigenfunctions and the simplicity of $\Lambda_\infty$ in this particular case. 

\subsection{Main results}

In order to present our main results, let us first introduce the notion of solutions to \eqref{inf-eigen} in geodesic spaces. Assume that $(\X, d)$ is a proper geodesic space. Let $\Omega\subsetneq \X$ be a bounded domain. We seek solutions of \eqref{inf-eigen-lambda} in the class of locally Lipschitz functions in $\Omega$. Let us briefly clarify our definitions of solutions below; see Section \ref{sec:def} for more precise descriptions. 

As usual, our definition can be divided into a supersolution part and a subsolution part. 
When defining a supersolution $u$, we require it to fulfill the supersolution properties for both the eikonal and the infinity Laplace equation. The former can be simply defined by
\beq\label{eikonal super}
|\nabla^- u|\geq \lambda u 
\eeq
in $\Omega$, where $|\nabla^- u|$ denotes the subslope of $u$, given by 
\[
|\nabla^- u|(x)=\limsup_{y\to x} \frac{\max\{u(x)-u(y), 0\}}{d(x, y)}.
\]
See also the definitions of slope $|\nabla u|$ and superslope $|\nabla^+ u|$ respectively in \eqref{slope eq} and \eqref{semislope eq}.

One can analogously use the subslope to define subsolutions and solutions of the eikonal  equation. This type of solutions of eikonal equations is called Monge solutions. Such a notion is studied in the Euclidean space \cite{NeSu, BrDa} and is introduced in \cite{LShZ} for general metric spaces. See \cite{LShZ} also for the equivalence with other notions of metric viscosity solutions proposed in \cite{AF, GHN, GaS2, GaS}. For the reader's convenience, we include several basic results on the eikonal equation in metric spaces in Appendix \ref{sec:app}.

Regarding the $\infty$-Laplacian supersolution property (or $\infty$-superharmonicity), we adopt the characterization of comparison with cones from below. This requires that an $\infty$-superharmonic function $u$ satisfy 
\[
\min_{\ol{\O}} (u-\phi)\geq \min_{\partial \O} (u-\phi)
\] 
for any open subset  $\O\subset\subset \Omega$ (i.e., $\ol{\O}\subset \Omega$)  and any cone function 
\beq\label{cone below}
\phi=a+\kappa d(\hat{x}, \cdot)\quad \text{in $\Omega$}
\eeq
with $a\in \R$, $\kappa\leq 0$ and $\hat{x}\in \Omega\setminus \O$. Consult \cite{Ju, JuSh} for more details on the properties of comparison with cones in connection with the absolute minimizing Lipschitz extensions in metric spaces. 

We consequently call $u$ a supersolution to \eqref{inf-eigen-lambda} if it satisfies both \eqref{eikonal super} and comparison with cones from below in $\Omega$. This notion looks very different from the usual viscosity supersolutions, since it is not defined pointwise by means of test functions. 
In contrast, we need to define subsolutions of \eqref{inf-eigen-lambda} pointwise and 
choose an appropriate class of test functions. While in the Euclidean case one can test a subsolution at $x_0\in \Omega$ by a $C^2$ function $v$ satisfying $-\Delta_\infty v(x_0)>0$ so as to obtain 
\beq\label{def sub eq eucl}
|\nabla v(x_0)|\leq \lambda u(x_0),
\eeq
finding the corresponding test class in general metric spaces is however not straightforward. We overcome the difficulty by adopting the class of $\infty$-superharmonic functions introduced above to test the candidate function in a strict manner. 

We say that $u$ is a subsolution of \eqref{inf-eigen-lambda} if whenever there exist an $\infty$-superharmonic function $v$ and $x_0\in \Omega$ such that $u-v$ attains a strict local maximum at $x_0$, we have
\beq\label{def sub eq}
\lim_{r\to 0+}\inf_{B_r(x_0)} |\nabla^- v|\leq \lambda u(x_0).
\eeq
The left hand side above looks slightly complicated. Actually \eqref{def sub eq} reduces to \eqref{def sub eq eucl} in $\R^n$ for a test function $v\in C^2(\Omega)$. In our current setting, the lower semicontinuous envelope of the subslope needs to be utilized due to the lack of smoothness of $v$. 

Our definitions of supersolutions and subsolutions of \eqref{inf-eigen-lambda} prove to be an appropriate generalization of those in the Euclidean case. Indeed, we can show the equivalence between both types of definitions in the case when $\X=\R^n$ based on the results in \cite{CEG,  ArCrJu, JuSh}; see Section \ref{sec:eucl} for details. 

Moreover, our new notions enable us to solve the $\infty$-eigenvalue problem in general geodesic spaces. Using our notion of supersolutions of \eqref{inf-eigen-lambda}, we define the $\infty$-eigenvalue in $\Omega$ by 
\beq\label{def-inf-eigenvalue}
\Lambda=\sup\left\{\lambda\in \R: \text{there exists a locally Lipschitz positive supersolution of \eqref{inf-eigen-lambda}}\right\}.
\eeq
This value turns out to coincide with \eqref{inf-eigenvalue general}, 
which is consistent with the Euclidean result. 


\begin{thm}[$\infty$-eigenvalue]\label{thm eigenvalue}
Suppose that $(\X, d)$ is a proper geodesic space and $\Omega\subsetneq \X$ is a bounded domain. Let $\Lambda_\infty>0$ be given by \eqref{inf-eigenvalue general}. 
Let $\Lambda$ be defined by \eqref{def-inf-eigenvalue}. Then $\Lambda=\Lambda_\infty$ holds.
\end{thm}

Our proof of Theorem \ref{thm eigenvalue}, which is elaborated in Section \ref{sec:eigenvalue}, consists of two steps. We first verify that 
\beq\label{special super}
u_{dist}(x)=\Lambda_\infty d(x, \pO),\quad \text{$x\in \Oba$,}
\eeq
is a supersolution of \eqref{inf-eigen}. This implies immediately that $\Lambda\geq \Lambda_\infty$.
The reverse inequality is shown by proving the non-existence of supersolutions of \eqref{inf-eigen-lambda} when $\lambda>\Lambda_\infty$. It actually follows from a fundamental property of supersolutions at any $x_0\in \Omega$ satisfying
\beq\label{center pt}
d(x_0, \pO)=R_\infty.
\eeq
Here and in the sequel we call such a point an incenter of $\Omega$. In fact, for an incenter point $x_0$, by comparing any positive supersolution $u$ with the cone function $\phi$ in \eqref{cone below} with $a=u(x_0)$ and any $\kappa<-|\nabla^- u|(x_0)$, one can prove that 
\[
|\nabla^- u|(x_0)\leq u(x_0)/R_\infty=\Lambda_\infty u(x_0),
\]
which yields $\Lambda\leq \Lambda_\infty$ by \eqref{def-inf-eigenvalue}.

Our second main result is on the existence of positive solutions of \eqref{inf-eigen} satisfying the boundary condition \eqref{dirichlet}. As usual, by solutions of \eqref{inf-eigen}, we mean locally Lipschitz functions that are both supersolutions and subsolutions. 
The positive solutions of \eqref{inf-eigen} satisfying \eqref{dirichlet} are called $\infty$-eigenfunctions in our current setting. In order to obtain the existence, we adapt Perron's method by taking the pointwise infimum of all supersolutions under the constraint 
\beq\label{center max}
u=\max_{\ol{\Omega}} u=1 \quad \text{on $\M(\Omega)$},
\eeq
where the set $\M(\Omega)$ is the so-called high ridge of $\Omega$,  containing all incenters in $\Omega$, i.e., 
\beq\label{incenter}
\M(\Omega):=\left\{x\in \Omega: d(x, \pO)=\max_{\ol{\Omega}} d(\cdot, \pO)=R_\infty\right\}.
\eeq
In other words, we set, for any $x\in \Oba$, 
\beq\label{min-eigenfun}
u_\infty(x)=\inf\left\{u(x):  \text{$u$ is a positive supersolution of \eqref{inf-eigen} satisfying \eqref{center max}}\right\}.
\eeq
We then obtain the following result.  
\begin{thm}[Existence of $\infty$-eigenfunctions]\label{thm existence sol}
Suppose that $(\X, d)$ is a proper geodesic space and $\Omega\subsetneq \X$ is a bounded domain. Let $u_\infty: \Oba\to \R$ be defined by \eqref{min-eigenfun}. 
Then $u_\infty$ is continuous in $\Oba$ and is a positive solution of \eqref{inf-eigen} satisfying \eqref{dirichlet} and  \eqref{center max}.
\end{thm} 

The proof of Theorem \ref{thm existence sol} streamlines Perron's method; see for example \cite{CIL} for a general introduction and \cite{ArCrJu, LW1, LW2} for applications to infinity Laplace equations in the Euclidean space.  But some arguments need to be slightly adapted. One noteworthy issue is about the regularity of solutions. Instead of carrying out Perron's method 
in the class of semicontinuous sub- or supersolutions, we choose to construct solutions of \eqref{inf-eigen} directly in the class of locally Lipschitz functions, as shown in \eqref{min-eigenfun}.
This is possible because the property of comparison with cones implies 
local Lipschitz regularity. This result is presented in Lemma \ref{lem local lip}. In fact, in Lemma \ref{lem semicon} we prove more for an $\infty$-superharmonic function $u$ in $\Omega$: it satisfies 
\beq\label{slope regularity eq}
|\nabla u|=|\nabla^- u|
\eeq 
and the slope $|\nabla u|$ is upper semicontinuous in $\Omega$; see also \cite{LShZ} for remarks on \eqref{slope regularity eq} in relation to the semiconcavity regularity. 
We use the regularity result to obtain the local Lipschitz continuity of $u_\infty$ after showing its $\infty$-superharmonicity. We also include a result (Proposition \ref{prop harnack}) on Harnack's inequality for $\infty$-superharmonic functions to show that $u_\infty>0$ in $\Omega$. Our new version of Harnack's inequality, which applies to general metric spaces, generalizes the results in \cite{LiM1, Bh} in the Euclidean case. 

In view of Theorem \ref{thm existence sol} and \eqref{min-eigenfun}, we see that $u_\infty$ is the minimal positive solution of \eqref{inf-eigen} satisfying the boundary condition \eqref{dirichlet} and the constraint \eqref{center max}. The constraint \eqref{center max} essentially normalizes any eigenfunction $u$ so that 
\beq\label{normalization eq}
\|u\|_{L^\infty(\Omega)}=1.
\eeq
One may wonder whether $u_\infty$ is the only solution satisfying \eqref{dirichlet} and \eqref{normalization eq}. In general it fails to hold. We can generalize our method to find a solution under a partial constraint
\beq\label{partial center max}
u=\max_{\Oba} u=1 \quad \text{in $Y$}
\eeq
instead of \eqref{center max}, where $Y$ is a compact subset of $\M(\Omega)$. Letting 
\beq\label{partial min-eigenfun}
u_\infty^Y(x):=\inf\left\{u(x): \ \text{$u$ is a positive supersolution of \eqref{inf-eigen} satisfying \eqref{partial center max}}\right\},
\eeq
we can follow the same proof to show that $u_\infty^Y$ is also a solution of \eqref{inf-eigen}--\eqref{dirichlet}. In general, it happens that $u_\infty^Y\neq u_\infty$ when $\M(\Omega)$ is not a singleton and $Y\subsetneq \M(\Omega)$. In Section \ref{sec:nonunique}, we present a concrete example in a finite metric graph, which can be regarded as a refined counterpart of the Euclidean example in \cite{HSY}. 

In addition to the results above, we also show in Theorem \ref{thm principal} that, under an additional regularity assumption on $\Omega$, $\Lambda_\infty$ is indeed the principal $\infty$-eigenvalue by proving that it is the least $\lambda$ that admits a positive solution of \eqref{inf-eigen-lambda} satisfying \eqref{dirichlet}. Such a result is obtained in the Euclidean space \cite[Theorem 3.1]{JLM1} via a comparison principle. In our current setting, we are not able to get a general comparison principle because of the absence of measure structure that is needed for generalization of the Crandall-Ishii lemma. However, we can still guarantee the minimality of $\Lambda_\infty$ by comparing a subsolution and a specific distance-based supersolution under the extra assumption on $\Omega$. 

The rest of the paper is organized in the following way. In Section \ref{sec:def}, we provide precise definitions of supersolutions and subsolutions of \eqref{inf-eigen} in geodesic spaces. Several important properties of $\infty$-superharmonic functions will also be studied including the Lipschitz regularity and Harnack's inequality. In Section \ref{sec:eigenprob}, we prove our main results, Theorem \ref{thm eigenvalue} and Theorem \ref{thm existence sol}, and give an example on metric graphs about the non-uniqueness of solutions. We also discuss the minimality of $\Lambda_\infty$ among all eigenvalues. For the reader's convenience, we include Appendix \ref{sec:app} to recall preliminaries on the eikonal equation in metric spaces.

\subsection*{Acknowledgements}
The work of the first author is supported by JSPS Grant-in-Aid for Scientific 
Research (No. 19K03574, No. 22K03396).  
The work of the second author is supported by JSPS Grant-in-Aid for Scientific Research (No. 20K03598). This work was also supported in part by funding (Grant No. 205004) from Fukuoka University.

\section{Definition and properties of solutions}\label{sec:def}

In this section we give a generalized notion of solutions to \eqref{inf-eigen-lambda} in a proper geodesic space for $\lambda\in \R$.  It is well known that the comparison with cones (cf. \cite{ArCrJu}) can be employed to characterize the $\infty$-harmonic functions; we refer to \cite{Ju, JuSh} for generalization in general metric spaces. We recall the definition of super- and subsolutions of 
\beq\label{inf-lap}
-\Delta_\infty u=0 \quad \text{in $\Omega$}
\eeq
based on this property. 
\begin{defi}[$\infty$-superharmonic functions]\label{def inf-super}
Let $(\X, d)$ be a  proper geodesic space and $\Omega\subsetneq \X$ be a bounded domain.
 A  function $u: \Omega\to \R$ that is bounded from below is 
said to be $\infty$-superharmonic in $\Omega$ (or a supersolution of \eqref{inf-lap}) if it satisfies the following property of comparison with cones from below in $\Omega$: for any $\hat{x}\in \Omega$, any $a\in \R$, $\kappa\leq 0$ and any bounded open set $\O\subset\subset \Omega$ with $\hat{x}\in\Omega\setminus\O$, the condition 
\beq\label{comp cone1}
u\geq \phi\quad \text{on $\partial \O$}
\eeq
for $\phi$ given by \eqref{cone below} implies that 
\beq\label{comp cone2}
u\geq \phi\quad \text{in $\overline{\O}$.}
\eeq
\end{defi}

\begin{rmk}[Definition of $\infty$-(sub)harmonic functions]
We can also define $\infty$-subharmonic in a symmetric way. More precisely, we say that any $u: \Omega\to \R$ bounded from above is $\infty$-subharmonic if $-u$ is $\infty$-superharmonic. In addition, $u$ is said to be $\infty$-harmonic if it is both $\infty$-superharmonic and $\infty$-subharmonic. 
\end{rmk}

We next provide an immediate consequence of Definition \ref{def inf-super}, which will be used later. 
Let us recall that for any $x\in \X$ and a locally Lipschitz function $u$, the local slope of $u$ at $x$ is given by
\beq\label{slope eq}
|\nabla u|(x)=\limsup_{y\to x} \frac{|u(y)-u(x)|}{d(x, y)}
\eeq
and the sub- and superslopes of $u$ at $x$ are defined to be
\beq\label{semislope eq}
|\nabla^\pm u|(x)=\limsup_{y\to x} \frac{[u(y)-u(x)]_{\pm}}{d(x, y)},
\eeq
where $[a]_\pm=\max\{\pm a, 0\}$ for any $a\in \R$.

\begin{lem}[Comparison with special cone functions]\label{lem circular}
Suppose that $(\X, d)$ is a  proper geodesic space and $\Omega\subsetneq \X$ is a bounded domain. Let $u : \Omega\to \mathbb R$ be $\infty$-superharmonic in $\Omega$, 
that is, $u$ is bounded from below and obeys the comparison with cones from below in $\Omega$.  Let $\Omega'\subset\subset \Omega$ be an open bounded set and $x_0\in \Omega'$. 
For $\kappa\leq 0$, if $u\geq \phi$ on $\partial \Omega'$, where 
\beq\label{circular1}
\phi(x):= u(x_0)+\kappa d(x_0, x), \quad x\in \Omega. 
\eeq
then $u\geq \phi$ in $\Omega'$. In addition, if $u$ is locally Lipschitz in $\Omega$, then for any $r>0$ satisfying $B_r(x_0)\subset\subset \Omega$, there exists $x_r\in \partial B_r(x_0)$ such that
\[
u(x_0)-u(x_r)\geq |\nabla^- u|(x_0) r.
\]
\end{lem}
\begin{proof}
This result follows directly from Definition \ref{def inf-super} with $\O=\Omega'\setminus \{x_0\}$ and the cone function taken to be $\phi$ as in \eqref{circular1}. 
We then can use the comparison with cones to get  $u\geq \phi$ in $\Omega'$.
 
To show the second statement, we assume by contradiction that it fails to hold, which yields existence of $\sigma>0$ small such that  
\[
u(x_0)-u(x)\leq (|\nabla^- u|(x_0)-\sigma) r
\]
for all $x\in \partial B_r(x_0)$. We may apply the previous result to show 
\[
u(x_0)-u(x)\leq (|\nabla^- u|(x_0)-\sigma) d(x_0, x)
\]
for all $x\in B_r(x_0)$. This is clearly a contradiction by the definition of $|\nabla^- u|$. 
\end{proof}

We next show that $\infty$-superharmonic 
 functions in general geodesic spaces are actually locally Lipschitz continuous. 
Related regularity results in the Euclidean case involving Harnack's inequality can be found in \cite{LiM1, Bh}. \begin{lem}[Local Lipschitz continuity]\label{lem local lip}
Suppose that $(\X, d)$ is a  proper geodesic space and $\Omega\subsetneq \X$ is a bounded domain.
Let $u : \Omega\to \mathbb R$ be $\infty$-superharmonic 
 in $\Omega$, that is, $u$ is bounded from below 
and obeys the comparison with cones from below in $\Omega$.  Then, $u$ is locally Lipschitz in $\Omega$. 
More precisely, for any $x \in \Omega$, $y \in B_r(x)$ with $0<2r < d(x, \partial \Omega)$, 
\beq\label{local lip est}
|u(x)-u(y)| \le \frac{\max\{u(x), u(y)\}-\inf_\Omega u}{r} d(x,y).
\eeq
A symmetric result  for an $\infty$-subharmonic function $u$ in $\Omega$; in this case, $u$ is still locally Lipchitz in $\Omega$ and satisfies \eqref{local lip est} with $u$ replaced by $-u$. 
\end{lem}
\begin{proof}
Suppose that $u$ is a supersolution that is bounded from below in $\Omega$.
Let us take $x \in \Omega$ and any $s>0$ satisfying $d(x, \partial \Omega)>s$. We consider a cone function defined by 
\[
\phi(z) := u(x) - \frac{u(x)-\inf_\Omega u}{s} d(x, z). 
\]
Then, $u \ge \inf_\Omega u \equiv \phi$ holds on $\partial B_s(x)$. 
Therefore, by Lemma \ref{lem circular} with $x_0=x$ and $\Omega'=B_s(x)$, we obtain
\begin{equation} \label{eq:Lip-1}
u(y) - u(x) \ge - \frac{u(x)-\inf_\Omega u}{s} d(x, y)
\end{equation}
whenever $y \in B_{s}(x)$ and $s<d(x, \partial \Omega)$. 

Let us take $r>0$ small with $2r < d(x, \partial \Omega)$ and $y \in B_r(x)$. Then \eqref{eq:Lip-1} immediately implies that 
\begin{equation} \label{eq:Lip-3}
u(y) - u(x) \ge - \frac{u(x)-\inf_\Omega u}{r} d(y, x).
\end{equation}
On the other hand, since 
$
d(y, \partial \Omega) \ge d(x, \partial \Omega) - d(x, y) > r, 
$
it follows from \eqref{eq:Lip-1} again that
\begin{equation} \label{eq:Lip-2}
u(x) - u(y)  \ge -\frac{u(y)-\inf_\Omega u}{r} d(x, y). 
\end{equation}
We conclude the proof by combining \eqref{eq:Lip-3} and \eqref{eq:Lip-2}. 
\end{proof}

A further regularity property of $\infty$-harmonic functions is as follows. 

\begin{lem}[Slope regularity of $\infty$-harmonic functions]\label{lem semicon}
Suppose that $(\X, d)$ is a  proper geodesic space and $\Omega\subsetneq \X$ is a bounded domain.
Let $u : \Omega\to \mathbb R$ be $\infty$-superharmonic 
in $\Omega$. 
Then, in $\Omega$ the slope and subslope of $u$ coincide, i.e., \eqref{slope regularity eq} holds, and they are both upper semicontinuous. A symmetric result  for an $\infty$-subharmonic function $u$ in $\Omega$, that is, 
$|\nabla u|$ and $|\nabla^+ u|$ coincide and are upper semicontinuous in $\Omega$. 
\end{lem}
\begin{proof}
We only consider the case when $u$ is $\infty$-superharmonic, since the argument is symmetric for $\infty$-subharmonic functions. 

By Lemma \ref{lem local lip}, $u$ is locally Lipschitz in $\Omega$. 
Fix $x_0\in \Omega$ arbitrarily. 
It is clear that
\beq\label{semicon eq2}
|\nabla u|(x_0)=\max\{|\nabla^+ u|(x_0), |\nabla^- u|(x_0)\}.
\eeq
In what follows let us show that 
\beq\label{semicon eq1}
|\nabla^- u|(x_0)\geq |\nabla^+ u|(x_0).
\eeq

For any $\vep>0$ small, we can find $r>0$ such that 
\[
u(x)\geq u(x_0)-(|\nabla^- u|(x_0)+\vep)d(x, x_0)
\]
for any $x\in B_r(x_0)$. Let us take 
\[
m_{\vep, r}=u(x_0)-(|\nabla^- u|(x_0)+\vep)r.
\]
For any $z\in B_r(x_0)$ close to $x_0$, let 
\[
\phi_z(x):=u(z)-{u(z)-m_{\vep, r}\over r_z} d(x, z)
\]
for $x\in B_r(x_0)$, where we set $r_z=r-d(x_0, z)$.
It is not difficult to see that $u\geq \phi_z$ on $\partial B_{r_z}(z)$. We apply Lemma \ref{lem circular} to get 
\beq\label{semicon eq3}
u\geq \phi_z\quad \text{ in $B_{r_z}(z)$}.
\eeq
In particular, for any $z$ sufficiently close to $x_0$ we have $u(x_0)\geq \phi_z(x_0)$, which yields
\[
u(z)-u(x_0)\leq {u(z)-m_{\vep, r}\over r_z} d(x_0, z)
\]
and therefore
\[
{[u(z)-u(x_0)]_+\over d(x_0, z)}\leq \left|{u(z)-m_{\vep, r}\over r_z}\right|\leq {|u(z)-u(x_0)|\over r-d(x_0, z)}+(|\nabla^- u|(x_0)+\vep){r\over r-d(x_0, z)}.
\]
Sending $d(z, x_0)\to 0$ and then $\vep\to 0$, we are led to \eqref{semicon eq1}. 
We immediately obtain $|\nabla^- u|=|\nabla u|$ in $\Omega$ due to \eqref{semicon eq2} and the arbitrariness of $x_0\in \Omega$. 

We now show the upper semicontinuity of $|\nabla^- u|$. It follows from \eqref{semicon eq3} again that 
\[
u(z)-u(x)\leq u(z)-\phi_z(x)={u(z)-m_{\vep, r}\over r_z} d(x, z)
\]
for all $x$ near $z$, which implies that
\[
|\nabla^- u|(z)\leq {u(z)-m_{\vep, r}\over r-d(x_0, z)}={u(z)-u(x_0)+(|\nabla^- u|(x_0)+\vep)r\over r-d(x_0, z)}.
\]
Letting $z\to x_0$, we have
\[
\limsup_{z\to x_0}|\nabla^- u|(z)\leq |\nabla^- u|(x_0)+\vep. 
\]
We complete our proof of the upper semicontinuity of $|\nabla^- u|$ in $\Omega$ by noticing that $\vep>0$ and $x_0\in \Omega$ are arbitrary. 
\end{proof}

In addition, we generalize, in the context of general geodesic spaces, a result in \cite{LiM1, Bh, Lindbook} on Harnack's inequality for the $\infty$-Laplace equation. 

\begin{prop}[Harnack's inequality]\label{prop harnack}
Suppose that $(\X, d)$ is a proper geodesic space and $\Omega$ is an open subset of $\X$. 
Assume that $u : \Omega \to \mathbb R$ satisfies the comparison with cones from below. 
Assume that $u\geq 0$ in  $B_R(x_0)$ with $B_R(x_0) \subset \Omega$. 
Then, 
\beq\label{eq harnack1}
u(y) \le 3 u(x)
\eeq
for any $x, y \in B_r(x_0)$ and $r > 0$ with $4r < R$. In addition, if $\Omega$ is connected, $u$ is lower semicontinuous, nonnegative on $\Omega$, and $\sup_{\Omega} u > 0$, then $u>0$ in $\Omega$. 
\end{prop}
\begin{proof}
Fix $y\in B_r(x_0)$ arbitrarily. 
Let us define 
\[
\kappa := \min_{z \in \partial B_{3r}(y)} u(z)- u(y)
\]
and consider a cone function
\[
\phi:= u(y) + \frac{\min\{\kappa, 0\}}{3r} d(\cdot ,\ y). 
\]
It is clear that $u \ge \phi$ holds on $\partial B_{3r}(y)$. By Lemma \ref{lem circular}, we then have $u\geq \phi$ in $B_{3r}(y)$. 

If $\kappa \ge 0$, then for any $x \in B_r(x_0)$, 
 we have $\phi(x)=u(y)$ and thus $
u(x) \ge u(y)$, which immediately implies \eqref{eq harnack1}. 

If $\kappa <0$, for any $x \in B_r(x_0)$, we have
\[
\begin{aligned}
u(x) &\ge \phi(x) = u(y) + \frac{\kappa}{3r} d(x, y) \\ 
&= \left(1 - \frac{1}{3r} d(x, y) \right) u(y) +\frac{1}{3r}d(x, y)\min_{z \in \partial B_{3r}(y)} u(z) \ge \frac{u(y)}{3}.
\end{aligned}
\]
Hence, we obtain \eqref{eq harnack1} again. 

Let us prove the second statement. 
Since $u$ is lower semicontinuous, the set 
\[
\Omega' := \{x \in \Omega : u(x ) > 0\}
\] 
is open in $\Omega$. 
If $\Omega \ne \Omega'$, then since $\Omega$ is connected, there is a point $x \in \partial \Omega' \cap \Omega$. 
Consequently, we have $u(x) = 0$, which is a contradiction to the first statement. 
We thus have completed the proof. 
\end{proof}

We next turn to the definition of supersolutions of \eqref{inf-eigen-lambda}. 

\begin{defi}[Supersolutions of $\infty$-eigenvalue problem]\label{def super}
Let $\Omega$ be a domain in a proper geodesic space $(\X, d)$. A locally Lipschitz function $u$ in $\Omega$ is called a supersolution of \eqref{inf-eigen-lambda} if $u$ is $\infty$-superharmonic in $\Omega$ 
 and \eqref{eikonal super}
holds everywhere in $\Omega$.
\end{defi}

\begin{rmk}
Thanks to Lemma \ref{lem local lip}, we may drop the local Lipschitz condition in the definition above provided that $u$ is known to be bounded from below in $\Omega$. In particular, any nonnegative $\infty$-superharmonic function in $\Omega$ is locally Lipschitz .
\end{rmk}

The idea of adopting the subslope rather than the entire slope to define the so-called Monge solutions of eikonal-type equations stems from the work  \cite{NeSu} in the Euclidean space and is recently applied to general complete length spaces in \cite{LShZ}. We refer to \cite{AF, GaS2, GHN, GaS} for alternative viscosity approaches to Hamilton-Jacobi equations in metric spaces. 

In contrast to the notion of supersolutions, it is less straightforward to define subsolutions in a general metric space.  
We use the class of $\infty$-superharmonic functions itself to test locally the candidate function in a strict manner.

\begin{defi}[Subsolutions of $\infty$-eigenvalue problem]\label{def sub}
Let $\Omega$ be a domain in a proper geodesic space $(\X, d)$. A locally Lipschitz function $u$ in $\Omega$ is called a subsolution of \eqref{inf-eigen-lambda} if whenever there exist $x_0\in \Omega$, $r_0>0$ small and an $\infty$-superharmonic function $v$ in $B_{r_0}(x_0)\subset \Omega$ 
such that $u-v$ attains a strict local maximum at $x_0$, the inequality \eqref{def sub eq} holds.

A locally Lipschitz function $u$ in $\Omega$ is called a solution of \eqref{inf-eigen-lambda} if it is both a supersolution and a subsolution. 
\end{defi}

If $|\nabla^- v|$ is known to be lower semicontinuous at $x_0$, then \eqref{def sub eq} can be rewritten as $|\nabla^- v|(x_0)\leq \lambda u(x_0)$. However, in general we only have upper semicontinuity of $|\nabla^- v|$ due to Lemma \ref{lem semicon} and the lower semicontinuity of solutions may fail to hold. 

Let us construct more $\infty$-superharmonic functions for our later use. 

\begin{lem}[$\infty$-superharmonic functions by composition]\label{lem composition}
Suppose that $(\X, d)$ is a  proper geodesic space and $\Omega\subsetneq \X$ is a bounded domain. Let $v$ be a positive $\infty$-superharmonic function in $\Omega$.  Let $h\in C^2((0, \sup_{\Omega} v))$ satisfy 
\beq\label{comp cond}
h'(v(x))>0,\quad h''(v(x))<0 \quad \text{for all $x\in \Omega$.}
\eeq
Then, for any cone function given by \eqref{cone below} with $a\in \R$, $\kappa\leq 0$ and $\hat{x}\in \Omega$, $h(v)-\phi$ cannot attain a local minimum in $\Omega\setminus \{\hat{x}\}$. In particular, $h(v)$ is $\infty$-superharmonic in $\Omega$.
\end{lem}
\begin{proof}
Suppose by contradiction that there exists a bounded open set $\O\subset\subset\Omega$ such that $h(v)-\phi$ attains a minimum at $x_0\in \O$ for a cone function given in \eqref{cone below} with $a\in \R$, $\kappa\leq 0$ and $\hat{x}\in \Omega\setminus \O$. By changing the value of $a$, we may assume that $h(v(x_0))=\phi(x_0)$ and $h(v)\geq \phi$ in $B_r(x_0)$ with $r>0$ small such that $\hat{x}\notin B_r(x_0)$. 

 Then by assumptions, $h$ admits an inverse function $h^{-1}$, of $C^2$ class, near $v(x_0)$. It follows that $v(x_0)=h^{-1}(\phi(x_0))$ and $v\geq h^{-1}(\phi)$ in $B_r(x_0)$. In addition, noticing that $h^{-1}$ is strictly convex near $\phi(x_0)$, by letting $r>0$ further small if necessary, we get
$
 v\geq h^{-1}(\phi)>\psi_0
$
in $B_r(x_0)\setminus \{x_0\}$, where we define
 \[
 \psi_0(x):=v(x_0)+{\kappa\over h'(v(x_0))}\left(d(x, \hat{x})-d(x_0 , \hat{x})\right),  \quad x\in B_r(x_0).
 \] 

 On the other hand, we have $v(x_0)=\psi_0(x_0)$. 
We therefore can take a cone function $\psi_\vep=\psi_0+\vep$ with apex at $\hat{x}\notin B_r(x_0)$ and $\vep>0$ sufficiently small so that $\psi_\vep\leq v$ on $\partial B_r(x_0)$ but $\psi_\vep(x_0)>v(x_0)$. 
This means that $v$ fails to obey the comparison with cones from below, which is clearly a contradiction to the assumption that $v$ is $\infty$-superharmonic. 
\end{proof}

Since $|\nabla^- h(v)|=h'(v) |\nabla^- v|$ holds for any $h\in C^2(\R)$, the result above amounts to saying that any composite function $h(v)$ serves as a test function for subsolutions in $\Omega$ provided that $v$ is $\infty$-superharmonic, $|\nabla^- v|>0$ in $\Omega$ and $h$ satisfies \eqref{comp cond}.

\section{Eigenvalue and eigenfunctions}\label{sec:eigenprob}

In this section, we study the eigenvalue problem \eqref{inf-eigen} associated to the infinity Laplacian.  We generalize the notion of the radius of the maximal inscribed metric ball in $\Omega$; namely, we take $R_\infty$ as in \eqref{inrad}, 
which can also be expressed by 
\[
R_\infty=\max_{x\in \Oba} d(x, \pO).
\]
We introduce a notion of the principal eigenvalue and show that it is indeed the value $\Lambda_\infty$ given in \eqref{inf-eigenvalue general}. We later provide a definition and some properties of the corresponding eigenfunction. 

\subsection{The eigenvalue}\label{sec:eigenvalue}

Let us begin with our notion of the principal eigenvalue associated to the $\infty$-Laplacian in geodesic spaces. 
\begin{defi}[$\infty$-eigenvalue]
Let $(\X, d)$ be a proper geodesic space and $\Omega\subsetneq \X$ be a bounded domain. The value $\Lambda\in \R$ given by \eqref{def-inf-eigenvalue} is called the principal eigenvalue for the $\infty$-Laplacian in $\Omega$ with the Dirichlet condition \eqref{dirichlet}.
\end{defi}
Although here we call $\Lambda$ the principal eigenvalue, its minimality among all eigenvalues is not obvious. In general it is not clear to us whether there exists a positive solution of \eqref{inf-eigen-lambda} and \eqref{dirichlet} for some $\lambda<\Lambda$. We will prove the minimality of $\Lambda$ in Theorem \ref{thm principal} under an additional assumption on $\Omega$. 

Our first main result, Theorem \ref{thm eigenvalue}, states that the $\infty$-eigenvalue $\Lambda$ as in \eqref{def-inf-eigenvalue} coincides with $\Lambda_\infty$, the reciprocal of $R_\infty>0$ in \eqref{inrad}. In order to prove Theorem \ref{thm eigenvalue}, we present the following result. 
\begin{prop}[Existence of typical supersolutions]\label{prop special super}
Suppose that $(\X, d)$ is a proper geodesic space and $\Omega\subsetneq \X$ is a bounded domain. Assume that $g\in C(\partial \Omega)$. 
For $\lambda>0$, set
 \beq\label{control eq}
u(x)=\min_{y\in \partial \Omega}\{g(y)+\lambda d(x, y)\}\quad \text{for $x\in \Oba$.}
 \eeq
Then  $u$ is $\infty$-superharmonic and $|\nabla^- u|=\lambda$ holds in $\Omega$. 
\end{prop}

\begin{proof}
The function $u$ given by \eqref{control eq} is known as the McShane-Whitney Lipschitz extension. By Theorem \ref{thm control}, we see that $u$ is Lipschitz in $\Oba$ and $|\nabla^- u|=\lambda$ in $\Omega$. It thus suffices to prove that $u$ is $\infty$-superharmonic in $\Omega$.  

Let $\O\subset\subset \Omega$ be a bounded open set and fix $\hat{x}\in \Omega\setminus \O$.  For any $a\in \R$ and $\kappa\leq 0$, let $\phi$ be given by \eqref{cone below}.
Suppose that \eqref{comp cone1} holds. 
 We aim to show that \eqref{comp cone2} holds. 

Assume by contradiction that this fails to hold. Then there exists $x_0\in \O$ such that 
\beq\label{eq dist super4}
\max_{x\in \ol{\O}} \left(\phi-u\right)(x)=(\phi-u)(x_0)=\mu
\eeq
for some $\mu>0$.
Due to the maximality at $x_0$, it is not difficult to see that
\beq\label{eq dist super5}
\lambda =|\nabla^- u|(x_0)\leq |\nabla^- \phi|(x_0)\leq |\nabla \phi|(x_0)=-\kappa.
\eeq

Note that there exists a Lipschitz curve $\ol{\gamma}$ with $\ol{\gamma}(0)=x_0$, $\ol{\gamma}(1)=\hat{x}$ and $\ell(\ol{\gamma})=d(x_0, \hat{x})$. Let 
\[
t_0=\inf\{t>0: \ol{\gamma}(t)\notin \O\}.
\]
Then $y_0=\ol{\gamma}(t_0)\in \partial \O$. It is clear that 
\beq\label{eq dist super6}
d(\hat{x}, x_0)=d(\hat{x}, y_0)+d(y_0, x_0).
\eeq
Moreover, we have $d(y_0, y)\leq d(y_0, x_0)+d(x_0, y)$ for all $y\in \partial \O$,
which, by \eqref{eq dist super4},  implies that 
\[
\begin{aligned}
u(y_0)&=\min_{y\in \partial\Omega} \{g(y)+\lambda d(y_0, y)\}\\
&\leq \lambda d(y_0, x_0)+\min_{y\in\partial \Omega} \{g(y)+\lambda d(x_0, y)\}\\
& \leq \lambda d(y_0, x_0)+a+\kappa d(\hat{x}, x_0)-\mu.\\
 \end{aligned}
\]
It follows from \eqref{eq dist super5} and \eqref{eq dist super6} that 
\[
u(y_0)\leq \kappa d(\hat{x}, x_0)-\kappa d(y_0, x_0)+a-\mu=a+\kappa d(\hat{x}, y_0)-\mu,
\]
which is clearly a contradiction to \eqref{comp cone1}. 
\end{proof}

\begin{rmk}[A distance-type supersolution]\label{rmk dist super}
Applying Proposition \ref{prop special super} with $g\equiv 0$ on $\pO$, we can see that, for any $\lambda\geq 0$,
$u=\lambda d(\cdot, \pO)$ is $\infty$-superharmonic. Also, since $\Lambda_\infty d(\cdot, \pO)\leq 1$ in $\Omega$, $u$ satisfies 
\[
|\nabla^- u|\geq \lambda\geq \lambda \Lambda_\infty d(\cdot, \pO) =\Lambda_\infty u
\]
in $\Omega$.  In particular, we see that $u_{dist}$ given by \eqref{special super} is a supersolution of \eqref{inf-eigen}. \end{rmk}

We now proceed to the proof of Theorem \ref{thm eigenvalue}.

\begin{proof}[Proof of Theorem \ref{thm eigenvalue}]
In view of Remark \ref{rmk dist super}, it is clear that $\Lambda_\infty\leq \Lambda$. Let us now prove $\Lambda\leq \Lambda_\infty$. Suppose that  there exists a locally Lipschitz positive supersolution $u$ of \eqref{inf-eigen-lambda} for some $\lambda>0$. Let $x_0\in \M(\Omega)$ be an incenter of $\Omega$, which satisfies \eqref{center pt}.

Fix $\vep>0$ arbitrarily small. Noticing that $u\geq \phi$ holds on $\partial B_s(x_0)$ with $s=R_\infty-\vep$, where 
\[
\phi(x):=u(x_0)-{u(x_0)\over s} d(x, x_0)\quad \text{for $x\in \Omega$},
\]
by Lemma \ref{lem circular}, we have 
\[
|\nabla^- u|(x_0)\leq {u(x_0)\over R_\infty-\vep}.
\]
Letting $\vep\to 0$, we end up with
\beq\label{incenter sub}
|\nabla^- u|(x_0)\leq \Lambda_\infty u(x_0), 
\eeq
which implies that $\lambda\leq \Lambda_\infty$. It then follows from \eqref{def-inf-eigenvalue} that $\Lambda\leq \Lambda_\infty$. 
\end{proof}

\subsection{Existence of eigenfunctions}\label{sec:eigenfunction}

Let us now investigate the existence of eigenfunctions. 
In this section, we aim to prove Theorem \ref{thm existence sol}, which states that $u_\infty$ defined by \eqref{min-eigenfun} is a solution of \eqref{inf-eigen} and \eqref{dirichlet}. 
We remark that $u_\infty$ is well defined, since by \ref{rmk dist super}, the function class for the inifimum in \eqref{min-eigenfun} is non-empty.

Let us now use Perron's method to prove Theorem \ref{thm existence sol}. We begin with the supersolution property of $u_\infty$. 

\begin{thm}[Supersolution property of infimum]\label{thm min-eigenfun super}
Suppose that $(\X, d)$ is a proper geodesic space and $\Omega\subsetneq \X$ is a bounded domain. Let $u_\infty: \Oba \to \R$ be defined by \eqref{min-eigenfun}. Then $u_\infty$  is a positive supersolution of \eqref{inf-eigen} satisfying \eqref{center max}. 
\end{thm}

To prove Theorem \ref{thm min-eigenfun super}, we prepare two results regarding the pointwise infima of supersolutions to the $\infty$-Laplace equation and to the eikonal equation respectively.

\begin{prop}[Supersolution preserving of infimum for $\infty$-Laplacian]\label{prop super infinity}
Suppose that $(\X, d)$ is a proper geodesic space and $\Omega\subsetneq \X$ is a bounded domain. 
Let $\S_I$ be a family of nonnegative $\infty$-superharmonic functions in $\Omega$. Then 
$w(x)=\inf\{u(x): u\in \S_I\}$ is also $\infty$-superharmonic in $\Omega$. 
\end{prop}
\begin{proof}
We only need to show that $w$ enjoys the property of comparison with cones from below. 
To see this, fix a bounded open set $\O\subset\subset \Omega$ and $\hat{x}\in \Omega\setminus \O$ and take a cone function $\phi$ as in \eqref{cone below}.
If $w \geq \phi$ on $\partial \O$, then by definition $u\geq \phi$ on $\partial \O$ for all $u\in \S_I$. It follows that $u\geq \phi$ in $\O$ for all $u\in \S_I$, which in turn implies that $w\geq \phi$ in $\O$. 
\end{proof}

\begin{prop}[Supersolution preserving of infimum for eikonal equation]\label{prop super eikonal}
Suppose that $(\X, d)$ is a proper geodesic space and $\Omega\subsetneq \X$ is a bounded domain. 
Let $\S_E$ be a family of nonnegative locally Lipschitz functions satisfying \eqref{eikonal super} in $\Omega$ for some $\lambda>0$. 
Assume that $w(x)=\inf\{u(x): u\in \S_E\}$ is locally Lipschitz in $\Omega$. Then $w$ 
is also a nonnegative function satisfying the same inequality in $\Omega$. 
\end{prop}
\begin{proof}

It suffices to prove that $W:=\log w$ satisfies 
$|\nabla^- W|\geq \lambda$ in $\Omega$.
Fix any $x_0\in \Omega$ and take $r>0$ arbitrarily small.  For any $u\in S_E$, 
letting $U=\log u$, we see that $U$ is locally Lipschitz 
and satisfies $|\nabla^- U|\geq \lambda$ in $B_r(x_0)$. It is clear that $W$ is the pointwise infimum over all such $U$. 

As shown in Proposition \ref{prop special super} (and in Theorem \ref{thm control}), the McShane-Whitney Lipschitz extension of $W$ given by
\[
\ul{W}(x)= \min_{y\in \partial B_r(x_0)}\{W(y)+\lambda d(x, y)\}, \quad x\in \ol{B}_r(x_0), 
\]
satisfies $|\nabla^- \ul{W}|=\lambda$ in $B_r(x_0)$. 
 Besides, it is easily seen that 
$\ul{W}\leq W\leq U$ on $\partial B_r(x_0)$.
We then can adopt the comparison principle, Theorem \ref{thm comparison monge} (or \cite[Theorem 4.2]{LShZ}), to deduce that $\ul{W}\leq U$ in $B_r(x_0)$,
where we recall that $U=\log u$ for each $u\in \S_E$. 

By taking the infimum over all such $u$, we obtain $\ul{W}\leq W$ in $B_r(x_0)$.
In particular, there exists $y_r\in \partial B_r(x_0)$ such that 
\[
W(x_0)\geq \ul{W}(x_0)\geq W(y_r)+\lambda d(x_0, y_r)
\]
and therefore 
\[
\frac{W(x_0)-W (y_r)}{d(x_0, y_r)}\geq \lambda.
\]
Sending $r\to 0$, we get
\[
|\nabla^- W|(x_0)\geq \limsup_{r\to 0}\frac{W(x_0)-W (y_r)}{d(x_0, y_r)}\geq \lambda.
\]
We complete the proof due to the arbitrariness of $x_0$ in $\Omega$. 
\end{proof}
 
We are now in a position to prove Theorem \ref{thm min-eigenfun super}.

\begin{proof}[Proof of Theorem \ref{thm min-eigenfun super}]
By definition, it is clear that $u_\infty$ satisfies \eqref{center max}. 
In view of Proposition \ref{prop super infinity}, we see that $u_\infty$ is $\infty$-superharmonic. 
Using Lemma \ref{lem local lip}, we obtain local Lipschitz continuity of $u_\infty$. By Proposition \ref{prop harnack} and the fact that $u_\infty=1$ on $\M(\Omega)$, we further deduce that $u_\infty>0$ in $\Omega$. Using Proposition \ref{prop super eikonal}, we have $|\nabla^- u_\infty|\geq \Lambda_\infty u_\infty$ in $\Omega$.  
Our proof  is thus complete. 
\end{proof}

We complete the proof of Theorem \ref{thm existence sol} by combining Theorem \ref{thm min-eigenfun super} with Theorem \ref{thm min-eigenfun sub} below, which states that 
$u_\infty$ is also a subsolution of \eqref{inf-eigen} in the sense of Definition \ref{def sub}.

\begin{thm}[Subsolution property of infimum of supersolutions]\label{thm min-eigenfun sub}
Suppose that $(\X, d)$ is a proper geodesic space and $\Omega\subsetneq \X$ is a bounded domain. Let $u_\infty: \Oba\to \R$ be defined by \eqref{min-eigenfun}. 
Then $u_\infty$ is continuous in $\Oba$ and is a subsolution of \eqref{inf-eigen} satisfying \eqref{dirichlet}. 
\end{thm}

\begin{proof}
We have shown in Theorem \ref{thm min-eigenfun super} that $u_\infty$ is a positive supersolution of  \eqref{inf-eigen}. 
Note that the definition of $u_\infty$, together with Remark \ref{rmk dist super}, yields 
$
u_\infty\leq u_{dist}=\Lambda_\infty d(\cdot, \pO) 
$ in $\ol{\Omega}$.
It is then easily seen that $u_\infty\in C(\Oba)$ and $u_\infty=0$ on $\partial \Omega$.

Let us focus on the subsolution property of $u_\infty$. 
Suppose by contradiction that $u_\infty$ is not a subsolution of \eqref{inf-eigen}. This means that there exist $r, \sigma >0$ small and $x_0\in \Omega$ with $B_{r}(x_0)\subset\subset \Omega$ such that 
\[
u_\infty(x)- u(x)< u_\infty(x_0)-u(x_0)=0
\]
for all $x\in B_r(x_0)\setminus \{x_0\}$, where $u$ is an $\infty$-superharmonic function satisfying
\beq\label{eigenfun sub1}
|\nabla^- u|\geq \Lambda_\infty u_\infty+\sigma \quad \text{in $B_{r}(x_0)$}.
\eeq
It follows that 
\beq\label{eigenfun rev1}
|\nabla^- u_\infty|(x_0)\geq |\nabla^- u|(x_0)\geq \Lambda_\infty u_\infty(x_0)+\sigma 
\eeq

As shown in the proof of Theorem \ref{thm eigenvalue}, we obtain \eqref{incenter sub} if $x_0\in \M(\Omega)$. Thus \eqref{eigenfun rev1} yields $x_0\notin \M(\Omega)$. 
We thus can take $r>0$ small such that $B_r(x_0)\cap \M(\Omega)=\emptyset$. 

We next take $\vep>0$ small such that 
\beq\label{eigenfun sub3}
u-\vep\geq u_\infty \quad \text{on $\Omega\setminus B_{r-\vep}(x_0)$.}
\eeq
We further take 
\[
\tilde{u}(x)=\begin{cases}
\min\{u_\infty(x), u(x)-\vep\} &\text{if $x\in B_r(x_0)$,}\\
u_\infty(x)  &\text{if $x\in \Oba\setminus B_r(x_0)$.}
\end{cases}
\]
It is clear that $\tilde{u}$ is continuous and positive in $\Omega$. One can also easily observe that $u_\infty$ satisfies \eqref{center max} and 
\beq\label{eigenfun sub5}
\tilde{u}(x_0)\leq u_\infty(x_0)-\vep.
\eeq

Let us below prove that $\tilde{u}$ is a supersolution of \eqref{inf-eigen}. 
By Proposition \ref{prop super infinity}, we deduce that $\tilde{u}$ is $\infty$-superharmonic in $B_r(x_0)$.  
Moreover, for any $x, y\in B_r(x_0)$ we have 
\[
\tilde{u}(x)-\tilde{u}(y)\geq \min\{u_\infty(x)-u_\infty(y), u(x)-u(y)\},
\]
which yields 
\[
|\nabla^- \tilde{u}|(x)\geq \min\{|\nabla^- u_\infty|(x), |\nabla^- u|(x)\}
\]
for all $x\in B_r(x_0)$.
Since $|\nabla^- u_\infty|\geq \Lambda_\infty u_\infty$ and  \eqref{eigenfun sub1} holds, it follows that
\beq\label{eigenfun sub4}
|\nabla^- \tilde{u}|\geq \Lambda_\infty \tilde{u}.
\eeq
in $B_r(x_0)$.
Noticing that $\tilde{u}=u_\infty$ in $\Omega\setminus B_{r-\vep}(x_0)$ due to \eqref{eigenfun sub3}, we thus see that \eqref{eigenfun sub4} holds in $\Omega$. 

It remains to verify that $\tilde{u}$ is $\infty$-superharmonic in $\Omega$. Suppose that there exist a bounded open set $\O\subset\subset \Omega$, $\hat{x}\in \Omega\setminus \O$ and a cone function as in \eqref{cone below} such that  $\tilde{u}\geq \phi$ on $\partial \O$. Since $\tilde{u}\leq u_\infty$, we have $u_\infty\geq \phi$ on $\partial\O$. Noticing that $u_\infty$ is $\infty$-superharmonic, we obtain $u_\infty\geq \phi$ in $\ol{\O}$ and in particular
\beq\label{eigenfun sub9}
\tilde{u}\geq \phi \quad \text{in $\ol{\O}\setminus B_{r-\vep}(x_0)$.}
\eeq
 It follows that $u-\vep\geq \phi$ on $\partial \O\setminus B_{r-\vep}(x_0)$,
which implies the same inequality on $\partial (B_{r-\vep}(x_0) \cap \O)$.

Since $u$ is $\infty$-superharmonic in $B_r(x_0)$, we obtain 
$u-\vep\geq \phi$ in $B_{r-\vep}(x_0)\cap \O$. Combining this with \eqref{eigenfun sub9}, we are led to $\tilde{u}\geq \phi$ in $\ol{\O}$.
 Hence, we conclude that $\tilde{u}$ is a positive supersolution of \eqref{inf-eigen}. 
Noticing that $\tilde{u}$ also satisfies \eqref{center max} and \eqref{eigenfun sub5},
we reach a contradiction to the definition of $u_\infty$ as in \eqref{min-eigenfun}.
Our proof is now complete. 
\end{proof}

\begin{rmk}[Partial incenter constraints]\label{rmk partial constraint}
Our argument above can be used to construct more solutions of \eqref{inf-eigen} when $\M(\Omega)$ is not a singleton. Recall that $\M(\Omega)$, defined by \eqref{incenter}, is the high ridge of $\Omega$. In fact, for any given compact subset $(\emptyset \neq)Y\subset \M(\Omega)$, replacing the condition  \eqref{center max} by \eqref{partial center max} in \eqref{min-eigenfun}, we can take 
$u_\infty^Y$ as in \eqref{partial min-eigenfun}.  
(It is clear that $u_\infty^Y=u_\infty$ when $Y=\M(\Omega)$.)
Then, following the proof of Theorem \ref{thm min-eigenfun super} and Theorem \ref{thm min-eigenfun sub}, we can prove in a similar way that $u_\infty^Y$ is also a positive solution of \eqref{inf-eigen}. In Section \ref{sec:nonunique} below, we present a concrete example on metric graphs to show that $u_\infty^Y$ and $u_\infty$ are really different in general. 
\end{rmk}

\subsection{Non-uniqueness under partial incenter constraints}\label{sec:nonunique}

In the Euclidean space an example is built \cite{HSY} in a dumbbell-shaped domain showing that in general there may be multiple linearly independent solutions to \eqref{inf-eigen} and \eqref{dirichlet}. In our general setting, this observation corresponds to the existence of  solutions under partial incenter constraints as described in Remark \ref{rmk partial constraint}.

In a similar manner to \cite{HSY}, for some particular domain $\Omega$ one can obtain at least one more solution to \eqref{inf-eigen} if the condition \eqref{center max} is weakened in the definition of $u_\infty$. We below present, on a metric graph, an analogue of the example in \cite{HSY} for non-uniqueness of solutions. 

\begin{example}\label{ex nonunique}
Let $\X=(\V, \E)$ be a finite graph with $\V=\{O\}\cup \{V_j\}_{j=0, \pm1, \pm2, \pm3}$ and $\E=\{e_j\}_{j=0, \pm1, \pm2, \pm3}$ satisfying
\beq\label{eq example uni1}
e_j=\begin{cases}
[O, V_{j}] &\text{for $j=0, \pm1$,} \\
[V_{+1}, V_j] & \text{for $j=+2, +3$},\\
[V_{-1}, V_j] & \text{for $j=-2, -3$}.
\end{cases}
\eeq
Then, equipped with the intrinsic metric $d$, $(\X, d)$ is clearly a geodesic space. Let $\Omega$ be the interior of $\X$; namely, 
\[
\Omega=\X\setminus \{V_j\}_{j=0, \pm1, \pm2, \pm3}, \quad \partial\Omega=\{V_j\}_{j=0, \pm1, \pm2, \pm3}.
\]
Assume that the length $\ell_j$ of each edge $e_j$ is given by
$
\ell_0=\ell_{\pm1}=\ell_{\pm2}=1, \ell_{\pm 3}=3.
$
We parametrize each $e_j$, using its length,  by $[0, \ell_j]$ with $t=0$ and  $t=\ell_j$ respectively corresponding to the left and right endpoints in the expression \eqref{eq example uni1}. 

One can show that $R_\infty=2$, $\Lambda_\infty=1/2$
and there are two incenters $P_{\pm}$ lying respectively on $e_{\pm3}$ with $d(P_{\pm}, V_{\pm3})=2$. We can also prove that $u_\infty$ given below is a solution of \eqref{inf-eigen}: 
\[
u_\infty(x)=\begin{dcases}
-{1\over 4}t+{1\over 4} & \text{for $x$ on $e_0$,}\\
{1\over 4}t+{1\over 4} & \text{for $x$ on $e_{\pm1}$,}\\
-{1\over 2}t+{1\over 2} & \text{for $x$ on $e_{\pm2}$,}\\
-{1\over 2}|t-1|+1 & \text{for $x$ on $e_{\pm3}$,}
\end{dcases}
\]
where $t$ represents the parameter for $x$ on each $e_j$ according to the parametrization given previously. See Figure \ref{fig:nonuni11} for an illustration of the function graph of $u_\infty$. 
Note that $u_{dist}$ in \eqref{special super} is not a solution and we have $u_\infty< u_{dist}$ on $e_{\pm1}$.

On the other hand, if we take $Y=\{P_+\}$, then we can construct another solution $u_\infty^Y$:
\[
u_\infty^Y(x)=\begin{dcases}
u_\infty(x) & \text{for $x$ on $e_j$ with $j=0, +1, +2, +3$,}\\
-{1\over 8}t+{1\over 4} & \text{for $x$ on $e_{-1}$,}\\
-{1\over 8}t+{1\over 8} & \text{for $x$ on $e_{-2}$,}\\
-{1\over 8}|t-1|+{1\over 4} & \text{for $x$ on $e_{-3}$.}
\end{dcases}
\]
See Figure \ref{fig:nonuni22} for the graph of $u_\infty^Y$. 

\begin{figure}[H]
\centering
\begin{minipage}{.5\textwidth}
  \centering
  \includegraphics[width=0.9\textwidth]{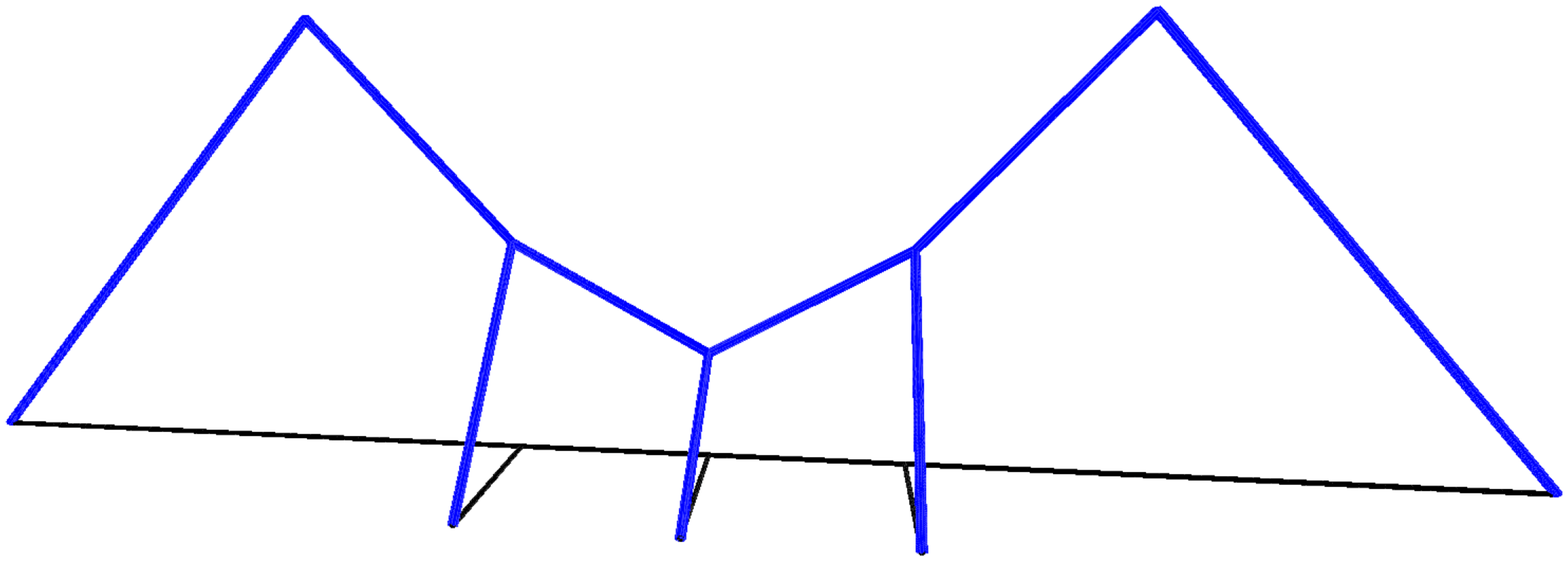}
  \vspace{-0.8cm}
  \caption{Graph of $u_\infty$}
  \label{fig:nonuni11}
\end{minipage}
\begin{minipage}{.5\textwidth}
  \centering
  \includegraphics[width=0.9\textwidth]{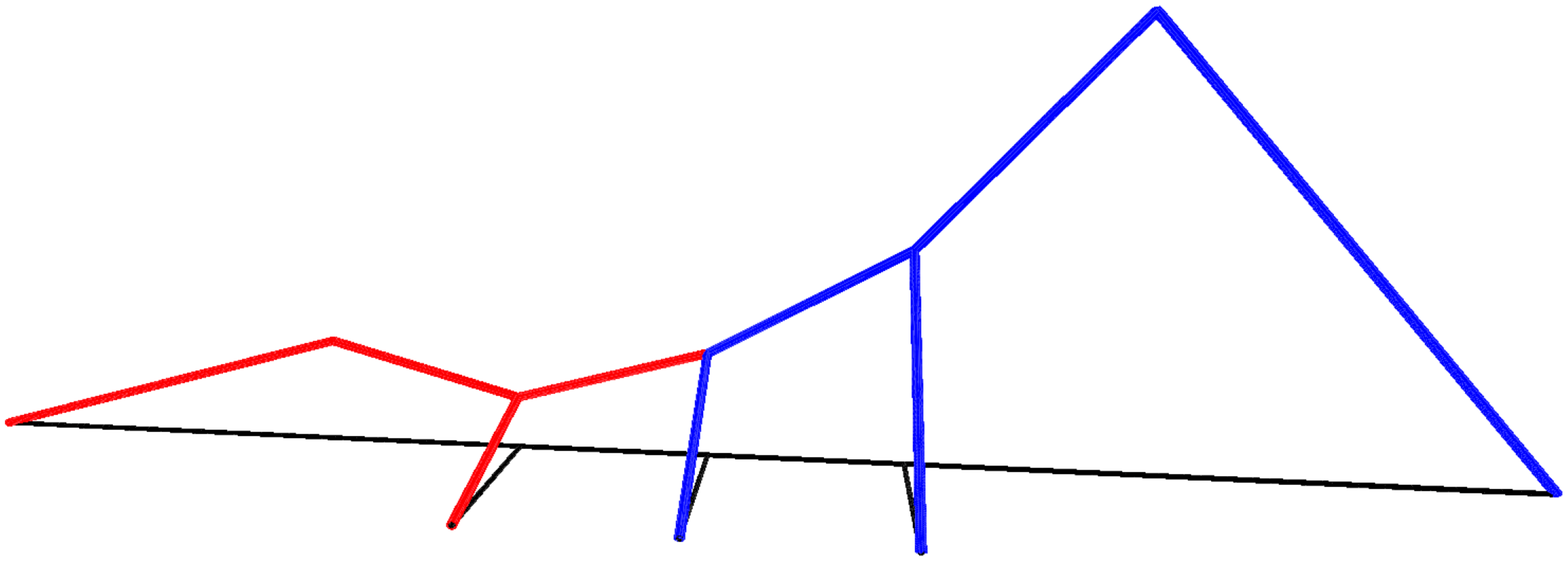}
  \vspace{-0.8cm}
  \caption{Graph of $u_\infty^Y$}
  \label{fig:nonuni22}
\end{minipage}
\end{figure}

Hence, in general we cannot expect uniqueness of solutions of \eqref{inf-eigen} up to a constant multiple. In the Euclidean spaces, the $\infty$-ground states (as the limits of $p$-eigenfunctions) are supposed to be symmetric in space. The function $u_\infty^Y$ above, which is not symmetric, thus corresponds to a non-variational solution on the metric graph. 
\end{example}

\subsection{Principal eigenvalue}

In this section, under an additional assumption on $\Omega$, we show that $\Lambda_\infty$ is indeed the smallest $\infty$-eigenvalue in the sense that any subsolution of \eqref{inf-eigen-lambda} associated to $\lambda<\Lambda_\infty$ is nonpositive in $\Omega$. 

In the Euclidean case, this result is proved by establishing a comparison principal for general subsolutions and supersolutions of \eqref{inf-eigen-lambda} \cite[Theorem 3.1]{JLM1}. 
We are however not able to implement the same machinery in our general setting due to the absence of measure structure, which is needed to invoke the Crandall-Ishii lemma. Instead, our proof consists in a comparison argument for an arbitrary subsolution and a specific distance-based supersolution under the assumption that $\Omega$ can be extended to a ``good'' domain. 
To be more precise, we introduce the following regularity assumption for a bounded domain $\O\subsetneq \X$.  Below let $\N_r(\O)$ denote the $r$-neighborhood of $\O$ in $\X$ for $r>0$; namely, 
\[
\N_r(\O):=\{x\in \X: d(x, \overline{\O})<r\}.
\]
\begin{enumerate}
\item[(A)] For any $x_0\in \O$ and $\vep>0$, there exist $r_0>0$ and a domain $E_{x_0}^{\vep}(\O)\subsetneq \X$ 
such that the following conditions hold:
\beq\label{ext1}
\O\subset E_{x_0}^{\vep}(\O)\subset \N_\vep(\O),
\eeq
\beq\label{ext2}
d(x_0, \partial E_{x_0}^{\vep}(\O))=d(x_0, \partial \O),
\eeq
\beq\label{ext3}
\begin{aligned}
&d(x, \partial E_{x_0}^{\vep}(\O))>d(x, \partial \O)\\ &\qquad\text{if $x\in B_{r_0}(x_0)\setminus \{x_0\}$ satisfies $d(x, \partial\O)= d(x_0, \partial\O)$.}
\end{aligned}
\eeq
\end{enumerate}
The assumption (A) requires certain smoothness of the domain $\O$. It can be better understood in the Euclidean space. Note that the sector region
\[
\O=\{(x_1, x_2)\in \R^2: x_1^2+x_2^2<1 \ \text{and either $x_1<0$ or $x_2<0$ holds}\}
\]
fails to satisfy this assumption; it is easily seen that, at each $x_0=(a, b)$ with $a, b<0$ and $a^2+b^2<1/2$, no matter how the extension $E_{x_0}^\vep(\O)$ is constructed, one cannot obtain \eqref{ext2} and \eqref{ext3} at the same time. On the other hand, we can show that a bounded domain $\O$ fulfills (A) for $\X=\R^n$ provided that it satisfies the exterior sphere condition at each of its boundary points; choosing a closest point $y_0\in \partial \O$ to $x_0$ and taking the exterior sphere tangent at $y_0$, we can extend $\O$ with $\partial E_{x_0}^\vep(\O)$ containing a portion of the exterior sphere near $y_0$.

\begin{thm}[Principal eigenvalue]\label{thm principal}
Suppose that $(\X, d)$ is a proper geodesic space. Let $\Omega\subsetneq \X$ be a bounded domain satisfying $\partial\Omega=\partial \Oba$.  Assume that for any $\delta>0$, there exists $\Omega_\delta\subsetneq\X$ such that $\Oba\subset \Omega_\delta\subset \N_\delta(\Omega)$, $\partial\Omega_\delta=\partial \Oba_\delta$, 
\beq\label{eq:minimality}
\min_{x\in \Oba} d(x, \partial \Omega_\delta)>0
\eeq
 and $\O=\Omega_\delta$ satisfies the regularity assumption (A).
Let $\lambda<\Lambda_\infty$. If $u\in C(\Oba)$ is a subsolution of \eqref{inf-eigen-lambda} and $u\leq 0$ on $\pO$, then $u\leq 0$ in $\Oba$.
\end{thm}

\begin{proof}
Assume by contradiction that there is a subsolution $u$ of \eqref{inf-eigen-lambda} such that $\sup_\Omega u>0$. Multiplying $u$ by an appropriate constant, we may assume that 
$u< 1$ in $\Oba$.

Let $\Omega_\delta$ be the domain described in the assumptions. Set 
\[
R_{\infty, \delta}:=\max_{x\in \overline{\Omega_\delta}} d(x, \pO_\delta),\quad 
\Lambda_{\infty, \delta}:={1\over R_{\infty, \delta}}.
\]
Let us first show that 
\beq\label{eq:minimal0}
R_{\infty, \delta}\to R_\infty\quad \text{as $\delta\to 0$.}
\eeq
Since $\Oba=\bigcap_{\delta>0} \N_\delta(\Omega)$, the assumption $\partial\Oba=\partial \Omega$ implies
\beq\label{eq:minimal1}
\partial\Omega=\partial\left(\bigcap_{\delta>0} \N_\delta(\Omega)\right). 
\eeq
Suppose that $x\in \Omega$ and $y\in \partial \Omega$ satisfy
$d(x, y)=d(x, \partial \Omega)=R_\infty$.
In view of \eqref{eq:minimal1}, we see that for any $\rho>0$, there exists $\delta>0$ small such that $B_\rho(y)\setminus \N_\delta(\Omega)\neq \emptyset$, 
which, by the condition $\Omega_\delta\subset \N_\delta(\Omega)$, yields $B_\rho(y)\setminus \Omega_\delta\neq \emptyset$. 
We thus can find $z\in B_\rho(y)\setminus \Omega_\delta$ such that 
$
d(y, B_\rho(y)\setminus \Omega_\delta)=d(y, z). 
$
Thus, any point, except $z$ itself, on a geodesic joining $y$ and $z$ must belong to $\Omega_\delta$. This shows that $z\in \partial \Omega_\delta$ and therefore 
$d(y, \partial\Omega_\delta)\leq \rho$.
We have shown that
\beq\label{bdry approx}
d(y, \partial\Omega_\delta)\to 0\quad \text{as $\delta\to 0$.}
\eeq
Since
\[
R_{\infty, \delta}\leq d(x, \partial \Omega_\delta)\leq d(x, y)+d(y, \partial\Omega_\delta)=R_\infty+d(y, \partial\Omega_\delta),
\]
by \eqref{bdry approx} we complete the proof of \eqref{eq:minimal0}. 
We therefore can fix $\delta>0$ and $0<\alpha<1$ sufficiently close to $1$ so that
\beq\label{minimality00}
\alpha\Lambda_{\infty, \delta}={\alpha \over R_{\infty, \delta}}>\lambda.
\eeq

In view of \eqref{eq:minimality}, we can next take $L>0$ large such that 
$
u\leq Ld(\cdot, \ \partial \Omega_\delta)^\alpha
$
in $\Oba$.
This implies that there exist $c\in (0, L]$ and $x_0\in \Omega$ such that 
\[
\max_{\Oba} \left(u-c d(\cdot, \partial \Omega_\delta)^\alpha\right)= u(x_0)-c d(x_0, \partial \Omega_\delta)^\alpha=0.
\]
For $\vep>0$, let
\[
h_{\vep}(t)=c t^\alpha+\vep (t-d(x_0, \partial \Omega_\delta))^2, \quad t>0.
\]
We take $\vep>0$ small such that $h_\vep''(d(x_0, \partial \Omega_\delta))<0$.

Let us next consider $E_{x_0}^\vep(\O)$ as in (A) for $\O=\Omega_\delta$ and  denote $\tilde{\Omega}_\vep:=E_{x_0}^\vep(\Omega_\delta).$
Let 
\[
\tilde{R}_{\vep}:=\max_{x\in \overline{\Omega_\vep}} d(x, \partial \tilde{\Omega}_\vep),\quad 
\tilde{\Lambda}_{\vep}:={1\over \tilde{R}_{\vep}}.
\]
Using the same argument in the proof of \eqref{eq:minimal0}, we deduce that $\tilde{R}_\vep\to R_{\infty, \delta}$ as $\vep\to 0$. We thus can adopt \eqref{minimality00} to obtain
\beq\label{eq:minimality2}
\alpha \tilde{\Lambda}_{\vep}>\lambda
\eeq
for any $\vep>0$ small.

Let $r_0>0$ be as in (A). Taking $r\in (0, r_0)$ small, we have
\beq\label{minimality1}
u(x)-h_\vep(d(x, \partial \Omega_\delta))< u(x_0)-h_\vep(d(x_0, \partial \Omega_\delta))
\eeq
for any $x\in B_r(x_0)\setminus\{x_0\}$ that satisfies
\beq\label{minimality2}
d(x, \partial \Omega_\delta)\neq d(x_0, \partial \Omega_\delta). 
\eeq
In addition, when $r>0$ is small, we get, for all $x\in B_r(x_0)$,
\[
h_\vep'(d(x, \partial \Omega_\delta))>0,\quad   h_\vep''(d(x, \partial \Omega_\delta))<0.
\]

Applying \eqref{ext1} and \eqref{ext2} with $\O=\Omega_\delta$, we have 
$d(\cdot, \partial\tilde{\Omega}_\vep)\geq d(\cdot, \partial\Omega_\delta)$ in $\Omega$ and
$d(x_0, \partial\tilde{\Omega}_\vep)= d(x_0, \partial\Omega_\delta)$.
We therefore can use \eqref{minimality1} to obtain
\beq\label{minimality3}
u(x)-h_\vep(d(x, \partial \tilde{\Omega}_\vep))< u(x_0)-h_\vep(d(x_0, \partial \tilde{\Omega}_\vep))=0
\eeq
for all $x\in B_r(x_0)$ satisfying \eqref{minimality2}. Noticing that \eqref{ext3} also holds with $\O=\Omega_\delta$, we see that \eqref{minimality3} actually holds for all $x\in B_r(x_0)\setminus\{x_0\}$. 

Denote $v_\vep=h_{\vep} (d(\cdot, \partial\tilde{\Omega}_\vep))$. By Lemma \ref{lem composition},  $v_\vep$ is $\infty$-superharmonic in $B_r(x_0)$. Moreover, in view of Remark \ref{rmk dist super} and the fact that $|\nabla^- d(\cdot, \partial \tilde{\Omega}_\vep)|=1$ in $B_r(x_0)$, we have 
\[
\begin{aligned}
|\nabla^- v_\vep|(x)&= \alpha c d(x, \partial \tilde{\Omega}_\vep)^{\alpha-1}|\nabla^- d(\cdot, \partial\tilde{\Omega}_\vep)|(x)+2\vep(d(x, \partial \tilde{\Omega}_\vep)-d(x_0, \partial \Omega_\delta))|\nabla^- d(\cdot, \partial\tilde{\Omega}_\vep)|(x)\\
&\geq \alpha c\tilde{\Lambda}_{\vep} d(x, \partial \tilde{\Omega}_\vep)^{\alpha}+2\vep(d(x, \partial \tilde{\Omega}_\vep)-d(x_0, \partial \Omega_\delta))\\
& = \alpha \tilde{\Lambda}_{\vep}v_\vep(x)-\alpha \vep \tilde{\Lambda}_{\vep} (d(x, \partial\tilde{\Omega}_\vep)-d(x_0, \partial \Omega_\delta))^2+2\vep(d(x, \partial \tilde{\Omega}_\vep)-d(x_0, \partial \Omega_\delta))
 \end{aligned}
\]
for any $x\in B_r(x_0)$, which by  \eqref{eq:minimality2} 
 and \eqref{minimality3} yields
\beq\label{minimality4}
\lim_{r\to 0}\inf_{B_r(x_0)}(|\nabla^- v_\vep|-\lambda u)\geq (\alpha \tilde{\Lambda}_{\vep}-\lambda) u(x_0)>0.
\eeq

On the other hand, since $v_\vep$ serves as a test function of $u$ at $x_0\in \Omega$, we apply the definition of subsolutions of \eqref{inf-eigen} to get 
\[
\lim_{r\to 0} \inf_{B_r(x_0)}(|\nabla^- v_\vep|- \lambda u)\leq 0, 
\]
which is a contradiction to \eqref{minimality4}.
\end{proof}

We remark that the condition $\partial \Omega=\partial \Oba$ in the result above is necessary to guarantee \eqref{eq:minimal0} and the continuity of the eigenvalue $\Lambda_\infty$ with respect to the extension $\Omega_\delta$; this condition is also assumed in \cite[Theorem 3.1]{JLM1} in the Euclidean case. In fact, if $\partial\Omega\neq \partial \Oba$, then \eqref{eq:minimal0} fails to hold in general, as shown in the following example. Let
\[
\Omega=\{(x_1, x_2)\in \R^2: x_1^2+x_2^2<1\}\setminus \{(x_1, x_2)\in \R^2:  0\leq x_1<1, \ x_2=0\}.
\]
Then one can easily see that $R_\infty=1/2$ and $R_{\infty, \delta}\geq 1$ for any extension $\Omega_\delta$ of $\Omega$ satisfying \eqref{eq:minimality}. For the same reason, we also assume that $\partial\Omega_\delta=\partial \Oba_\delta$ so that $\tilde{R}_\vep$ approximates $R_{\infty, \delta}$ as $\vep\to 0$.

\section{Consistency with the Euclidean case}\label{sec:eucl}

By Theorem \ref{thm eigenvalue}, we have seen that our $\infty$-eigenvalue is consistent with that in the Euclidean case. 
In this section we further show that our definition of eigenfunctions in geodesic spaces, as given in Definition \ref{def sub}, is also a generalization of the notion proposed in \cite{JLM1} in the Euclidean space. 

Let us recall the definition of viscosity supersolutions of \eqref{inf-eigen-lambda} in the Euclidean space, which is implicitly given in \cite{JLM1} as follows. See \cite{CIL} for definitions of viscosity solutions to general nonlinear elliptic equations. 
Let $\Omega$ be a bounded domain in $\R^n$ and $\lambda>0$. 
\begin{defi}[Definition of Euclidean viscosity solutions]\label{def eucl}
A locally bounded lower semicontinuous function $u: \Omega\to \R$ is called a viscosity supersolution of \eqref{inf-eigen-lambda} if whenever there exist $x_0\in \Omega$ and $\varphi\in C^2(\Omega)$ such that $u-\varphi$ attains a strict local minimum in $\Omega$ at $x_0$, both
\beq\label{vis eikonal super}
|\nabla \varphi(x_0)|\geq \lambda u(x_0)
\eeq
and
\beq\label{vis superhar}
-\Delta_\infty \varphi(x_0)\geq 0
\eeq
hold. A locally bounded upper semicontinuous function $u: \Omega\to \R$ is called a viscosity subsolution of \eqref{inf-eigen-lambda} if whenever there exist $x_0\in \Omega$ and $\varphi\in C^2(\Omega)$ such that $u-\varphi$ attains a strict local maximum in $\Omega$ at $x_0$, either 
\beq\label{vis eikonal sub}
|\nabla \varphi(x_0)|\leq \lambda u(x_0)
\eeq
or 
\[
-\Delta_\infty \varphi(x_0)\leq 0
\]
holds. A function $u\in C(\Omega)$ is called a viscosity solution of \eqref{inf-eigen-lambda} if it is both a viscosity supersolution and a viscosity subsolution. 
\end{defi}

\begin{rmk}\label{rmk viscosity2}
As is well known in the theory of viscosity solutions, one may use the semijets instead of the test functions to define super- and subsolutions. More precisely,  $u\in C(\Omega)$  is a viscosity supersolution (resp., subsolution) if for any $x_0\in \Omega$, we have
\[
\begin{aligned}
&|p|\geq \lambda u(x_0)\quad \text{ and } -\langle Xp, p\rangle \geq 0\\
&\left(\text{resp.,}\quad |p|\leq \lambda u(x_0)\quad \text{ or } -\langle Xp, p\rangle \leq 0\right)
\end{aligned}
\]
for every $(p, X)\in \ol{J}^{2, -}u(x_0)$ (resp., $(p, X)\in \ol{J}^{2, +}u(x_0)$). We refer to \cite{CIL} for a detailed introduction on the semijets $\ol{J}^{2, \pm}u(x_0)\subset \R^n\times \S^{n}$, where 
$\S^{n}$ denotes the set of all $n\times n$ real-valued symmetric matrices. 

The definition above is indeed consistent with the standard framework of viscosity solutions.  For a general fully nonlinear elliptic equation 
\[
F(x, u, \nabla u(x), \nabla^2 u)=0\quad \text{in $\Omega\subset \R^n$}, 
\]
where $F: \Omega\times \R\times \R^n\times \S^n\to \R$ is continuous, we define a viscosity supersolution (resp., subsolution) $u$ by demanding 
\[
F(x, u(x), p, X)\geq 0 \quad \left(\text{resp.,}\quad  F(x, u(x), p, X)\leq 0 \right)
\]
for all $(p, X)\in \ol{J}^{2, -} u(x)$ (resp., $(p, X)\in \ol{J}^{2, +} u(x)$). See again \cite{CIL} for details. 
\end{rmk}

The main result of this section is as follows. 

\begin{thm}[Equivalence of solutions]\label{thm eucl sol}
Let $\Omega\subset \R^n$ be a bounded domain and $\lambda>0$. Let $u$ be a positive locally Lipschitz function in $\Omega$. Then $u$ is a solution of \eqref{inf-eigen-lambda} in the sense of Definition \ref{def super} and Definition \ref{def sub} if and only if $u$ is a viscosity solution of \eqref{inf-eigen-lambda} as defined in Definition \ref{def eucl}.
\end{thm}

Before starting to prove this theorem, we first present the following equivalence result on $\infty$-superharmonic functions, which is essentially obtained in \cite{ArCrJu}.

\begin{lem}[Equivalence of $\infty$-superharmonic functions]\label{lem eucl cone}
Let $\Omega\subset \R^n$ be a bounded domain. Then $u$ is $\infty$-superharmonic in $\Omega\subset \R^n$ in the sense of Definition \ref{def inf-super} if and only if $u$ is a viscosity supersolution of \eqref{inf-lap}, that is, whenever there exist $x_0\in \Omega$ and $\varphi\in C^2(\Omega)$ such that $u-\phi$ attains a strict local minimum at $x_0$, the inequality \eqref{vis superhar} holds. 
\end{lem}
\begin{proof}
In \cite[Theorem 3.1]{CEG} and \cite[Theorem 4.13]{ArCrJu}, it is proved that $u$ is a viscosity supersolution of \eqref{inf-lap} if and only if $u$ satisfies the property of comparison with cones from below. Note that in these results each test cone $\phi$ in an open subset $\O$ 
is in the form of \eqref{cone below} with $a, \kappa\in \R$ and $\hat{x}\in \Omega\setminus \O$. 

In order to complete the proof, it suffices to show that in the Euclidean space this more restrictive condition  (with $\kappa\in \R$) of cone comparison is equivalent to the version with $\kappa\geq 0$ as stated in Definition \ref{def inf-super}. This can be found in \cite{JuSh}; see the remarks after Definition 2.3 in \cite{JuSh}, where a sufficient condition is given for the equivalence of both types of comparison with cones in more general metric spaces. 
\end{proof}

We first show the equivalence between the notions of supersolutions. 

\begin{prop}[Equivalence of supersolutions]\label{prop eucl super}
Let $\Omega\subset \R^n$ be a bounded domain and $\lambda>0$. Let $u$ be a positive continuous function in $\Omega$.  Then $u$ is a supersolution of \eqref{inf-eigen-lambda} in the sense of Definition \ref{def super} if and only if $u$ is a viscosity supersolution of \eqref{inf-eigen-lambda} as defined in Definition \ref{def eucl}.
\end{prop}

\begin{proof}
Let us first show ``$\Rightarrow$''. Suppose that $u$ is locally Lipschitz and $\infty$-superharmonic, and \eqref{eikonal super} holds in $\Omega$. Assume that $u-\varphi$ attains a strict local minimum at $x_0\in \Omega$ for some $\varphi\in C^2(\Omega)$.  By Lemma \ref{lem eucl cone}, we have \eqref{vis superhar}. Also, it is not difficult to see that 
\[
|\nabla\varphi|(x_0)|= |\nabla^-\varphi |(x_0)\geq |\nabla^- u|(x_0),
\]
which, together with \eqref{eikonal super}, immediately yields \eqref{vis eikonal super}. Hence, we conclude that $u$ is a viscosity supersolution of \eqref{inf-eigen-lambda}.

We next prove ``$\Leftarrow$''. Suppose that $u$ is viscosity supersolution of \eqref{inf-eigen-lambda}. Applying Lemma \ref{lem eucl cone} again, we deduce that $u$ satisfies the property of comparison with cones in $\Omega$. In particular, $u$ is locally Lipschitz by Lemma \ref{lem local lip}. It remains to show \eqref{eikonal super} in $\Omega$.  

Fix $x_0\in \Omega$ arbitrarily. Thanks to the $\infty$-superharmonicity of $u$, for any $\sigma>0$ small, we can find $r>0$ small such that 
\[
u(x)> u(x_0)-(|\nabla^- u|(x_0)+\sigma) d(x_0, x)
\]
for all $x\in \ol{B_r(x_0)}\setminus \{x_0\}$. It amounts to saying that $\phi_\sigma-u$ attains a strict maximum in $\ol{B_r(x_0)}$ at $x_0$, where 
\[
\phi_\sigma(x):=u(x_0)-(|\nabla^- u|(x_0)+\sigma) d(x_0, x), \quad x\in \ol{B_r(x_0)}. 
\]
Let us consider 
\[
\Phi_\vep(x, y)=\phi_\sigma(x)-u(y)-{|x-y|^2\over \vep}
\]
for $\vep>0$ and $x, y\in \ol{B_r(x_0)}$. Suppose that $(x_\vep, y_\vep)$ is a maximizer of $\Phi_\vep$ in $\ol{B_r(x_0)}\times  \ol{B_r(x_0)}$. By a standard argument of viscosity solutions, we have $x_\vep, y_\vep\to x_0$ as $\vep\to 0$. In particular, we get $x_\vep, y_\vep\in B_r(x_0)$ when $\vep>0$ is taken small.

Since 
\[
y\mapsto u(y)+{|x_\vep-y|^2\over \vep}-\phi_\sigma(x_\vep)
\]
attains a local minimum at $y=y_\vep$, we can apply the supersolution part of Definition \ref{def eucl} to deduce that
\beq\label{eucl super1}
 {2| x_\vep-y_\vep|\over \vep}\geq \lambda u(y_\vep).
\eeq
On the other hand, the maximality at $x=x_\vep$ of 
\[
x\mapsto \phi_\sigma(x)-u(y_\vep)-{|x-y_\vep|^2\over \vep}
\]
yields
\beq\label{eucl super2}
 {2|x_\vep-y_\vep|\over \vep}\leq |\nabla^- \phi_\sigma|(x_\vep)=|\nabla^- u|(x_0)+\sigma.
\eeq
Combining \eqref{eucl super1} and \eqref{eucl super2}, we are led to 
\[
|\nabla^- u|(x_0)+\sigma\geq \lambda u(y_\vep).
\]
Our proof is thus complete if we send $\vep\to 0$ and then $\sigma \to 0$ in the relation above. 
\end{proof}

We next show the equivalence for subsolutions, 
using a comparison-type argument.

\begin{prop}[Equivalence of subsolutions]\label{prop eucl sub}
Let $\Omega\subset \R^n$ be a bounded domain and $\lambda>0$. Let $u$ be a positive locally Lipschitz function in $\Omega$.  Then $u$ is a subsolution of \eqref{inf-eigen-lambda} in the sense of Definition \ref{def sub} if and only if $u$ is a viscosity subsolution of \eqref{inf-eigen-lambda} as defined in Definition \ref{def eucl}.
\end{prop}
\begin{proof}
We again begin with the proof of ``$\Rightarrow$''. Assume that $u$ is locally Lipschitz and there exist $x_0$ and $\varphi\in C^2(\Omega)$ such that $u-\varphi$ attains a strict local maximum at $x_0$. If 
$-\Delta_\infty \varphi(x_0)>0$,
then $-\Delta_\infty \varphi>0$ in $B_r(x_0)$
for some $r>0$ small. In view of Lemma \ref{lem eucl cone}, we know that $\varphi$ is $\infty$-superharmonic. 
It thus follows from Definition \ref{def sub} that 
\beq\label{eucl sub1}
\lim_{r\to 0+} \inf_{B_r(x_0)} \left(|\nabla^- \varphi|-\lambda u\right)\leq 0.
\eeq
This implies \eqref{vis eikonal sub}, since $|\nabla^- \varphi|=|\nabla \varphi|$ is continuous.

Let us show the reverse implication ``$\Leftarrow$''. Let $u$ be locally Lipschitz in $\Omega$. Assume that there exist $x_0\in \Omega$, $r>0$ small and a function $\varphi$ that is $\infty$-superharmonic in $B_r(x_0)$ such that $u-\varphi$ attains a strict maximum in $B_r(x_0)$ at $x_0$. It follows that 
\[
\max_{\partial B_r(x_0)} (u-\varphi)< u(x_0)-\varphi(x_0).
\]
 By Lemma \ref{lem eucl cone}, we see that $-\Delta_\infty\varphi\geq 0$ in $B_r(x_0)$ holds in the viscosity sense. We aim to get \eqref{eucl sub1}. 

By contradiction we assume that there exists $\sigma>0$ small such that 
$|\nabla^- \varphi|-\lambda u\geq \sigma$ in $\ol{B_r(x_0)}$. 
In particular, we see that $|\nabla\varphi|\geq \sigma$ holds in the viscosity sense in $B_r(x_0)$. 
Without loss of generality, we may also assume that $0<\varphi<1$ in $\ol{B_r(x_0)}$. For for $A>1$ and $\alpha<1$, we set $\varphi_h:=h(\varphi)$ with $h$ given by  
\[
h(t)=(At^{1\over\alpha}+1-A)^\alpha, \quad t>0.
\]
We approximate $\varphi$ by $\varphi_h$ in $B_r(x_0)$  with $A>1$ and $\alpha<1$ close to $1$, we have
\beq\label{eucl sub add2}
\max_{\partial B_r(x_0)} (u-\varphi_h)< \max_{\ol{B_r(x_0)}} (u-\varphi_h).
\eeq
Since 
\[
h'(t)=A(At^{1\over\alpha}+1-A)^{\alpha-1} t^{1-\alpha\over\alpha}>A^\alpha >1,
\]
\[
h''(t)={1\over \alpha}(\alpha-1)A(A-1) (At^{1\over\alpha}+1-A)^{\alpha-2} t^{1-2\alpha\over 2}<0
\]
hold for all $t=\varphi(x)$ with $x\in B_r(x_0)$, we see that 
 \beq\label{eucl sub add3}
 |\nabla^- \varphi_h|-\lambda u=h'(\varphi) |\nabla^- \varphi|-\lambda u>|\nabla^- \varphi|-\lambda u\geq \sigma,
 \eeq
 and
\beq\label{eucl sub4}
-\Delta_\infty \varphi_h>0
\eeq
hold in $B_r(x_0)$ in the viscosity sense. 
For each $\vep>0$, let us consider  
\[
\Psi_\vep(x, y)=u(x)-\varphi_h(y)-{|x-y|^2\over \vep}
\]
for $x, y\in \ol{B_r(x_0)}$. We can find a maximizer $(x_\vep, y_\vep)$ of $\Psi_\vep$ in $\ol{B_r(x_0)}\times  \ol{B_r(x_0)}$. As in the proof of Proposition \ref{prop eucl super}, we apply a standard argument for comparison principle of viscosity solutions as well as \eqref{eucl sub add2} to deduce that $x_\vep, y_\vep\to x_0$ as $\vep\to 0$ and therefore 
$x_\vep, y_\vep\in B_r(x_0)$ for $\vep>0$ sufficiently small. 

Noticing that 
\[
y\mapsto u(x_\vep)-\varphi_h(y)-{|x-y|^2\over \vep}
\]
attains a maximum at $y=y_\vep$, we apply \eqref{eucl sub add3} to get
\beq\label{eucl sub8}
{2|x_\vep-y_\vep|\over \vep}
\geq \lambda u(y_\vep)+\sigma. 
\eeq

We next adopt the Crandall-Ishii lemma \cite{CIL} to obtain $(p, X)\in \ol{J}^{2, +} u(x_\vep)$ and $(q, Y)\in \ol{J}^{2, -} \varphi_h(y_\vep)$ satisfying 
\beq\label{eucl sub2}
p=q={2(x_\vep-y_\vep)\over \vep},
\eeq
and 
\[
\begin{pmatrix}
X & 0\\
0& -Y
\end{pmatrix}
\leq {C\over \vep} \begin{pmatrix} I & -I\\ -I & I\end{pmatrix}.
\]
for some $C>0$. It follows that 
\beq\label{eucl sub3}
X\leq Y.
\eeq

Using the viscosity inequality \eqref{eucl sub4} at $y_\vep$ in the form of semijets as in Remark \ref{rmk viscosity2}, we are led to 
$
-\la Y q, q\ra>0,
$ 
which by \eqref{eucl sub2} and \eqref{eucl sub3} implies 
-$\la Xp, p\ra>0$.
Applying the alternative definition of subsolutions of \eqref{inf-eigen-lambda} with semijets on $u$, 
we thus get 
\beq\label{eucl sub7}
{2|x_\vep-y_\vep|\over \vep}=|p|\leq \lambda u(x_\vep).
\eeq
Combining \eqref{eucl sub7} and \eqref{eucl sub8}, we end up with 
\[
\lambda u(x_\vep)\geq \lambda u(y_\vep)+\sigma,
\]
We reach a contradiction by letting $\vep\to 0$. 
\end{proof}

\appendix

\section{Eikonal equation in metric spaces}\label{sec:app}
In this appendix, we present several results on the eikonal equation 
in metric spaces that are used in the proofs of Proposition \ref{prop special super} and Theorem \ref{thm min-eigenfun super}. As mentioned in the introduction, for $\lambda>0$ the equation $|\nabla u|-\lambda u=0$ related to the eigenvalue problem can be turned into the standard eikonal equation $|\nabla U|=\lambda$ via the transformation $U=\log u$.

We thus consider the eikonal equation
\beq\label{eikonal eq}
|\nabla u|=\lambda \quad \text{in $\Omega$}
\eeq
with the Dirichlet boundary condition $u=g$ on $\pO$,
where $\lambda>0$ and $g\in C(\pO)$ are given. We still assume that $(\X, d)$ is a proper geodesic space and $\Omega\subsetneq\X$ is a bounded domain. 

In \cite{LShZ}, the notion of Monge solutions of the Hamilton-Jacobi equations in the Euclidean space \cite{NeSu, BrDa} is also generalized for general length spaces. A locally Lipschitz function in $\Omega$ is called a Monge solution (resp., Monge supersolution, Monge subsolution) of \eqref{eikonal eq} if $|\nabla^- u|= \lambda$ (resp., $\geq \lambda$, $\leq \lambda$) in $\Omega$.

An optimal control interpretation is provided in \cite[Theorem 4.2]{GHN} to construct solutions of general eikonal equations in metric spaces. It is shown in \cite{LShZ} that such solutions are actually Monge solutions. For our particular purpose in this work, we below build a slightly different Monge solution $u$ satisfying $u\leq g$ on $\pO$, which is simply the celebrated McShane-Whitney extension in $\Oba$ with Lipschitz constant $\lambda$.

\begin{thm}[Construction of a Monge solution]\label{thm control}
Suppose that $(\X, d)$ is a proper geodesic space and $\Omega\subsetneq \X$ is a bounded domain. Assume that $\lambda>0$ and $g\in C(\pO)$. Let $u$ be given by \eqref{control eq}.
Then $u$ is Lipschitz in $\Oba$ and $|\nabla^- u|=|\nabla u|=\lambda$
holds in $\Omega$. 
\end{thm}

\begin{proof}
We first claim that, for any subdomain $\O$ with $\ol{\O}\subset \Omega$ and any $x\in \O$, 
\beq\label{dpp}
u(x)=\min_{z\in \ol{\O}} \{u(z)+ \lambda d(x, z)\}.
\eeq
For any $z\in \ol{\O}$, by \eqref{control eq} there exists $y_z\in \pO$ such that 
\beq\label{control eq2}
u(z)\geq g(y_z)+\lambda d(z, y_z). 
\eeq
It follows from \eqref{control eq} again that 
\beq\label{control eq3}
u(x)\leq g(y_z)+\lambda d(x, y_z)\leq g(y_z)+\lambda d(z, y_z)+\lambda d(x, z)\leq u(z)+\lambda d(x, z),
\eeq
which, due to the arbitrariness of $z\in \ol{\O}$ yields
\[
u(x)\leq \min_{z\in \ol{\O}} \{u(z)+ \lambda d(x, z)\}.
\]
On the other hand, we can find $y_x\in \pO$ such that 
\beq\label{control1}
u(x)\geq g(y_x)+\lambda d(x, y_x).
\eeq
Take a geodesic $\gamma$ connecting $x$ and $y_x$, i.e., $\gamma(0)=x$, $\gamma(1)=y_x$ and $\ell(\gamma)=d(x, y_x)$. There must exist a point of intersection $z_x$ of $\gamma$ and $\partial \O$, which satisfies 
\[
d(x, y_x)=d(x, z_x)+d(z_x, y_x).
\]
In view of \eqref{control1}, we thus can apply \eqref{control eq} once again to get
\[
u(x)\geq g(y_x)+\lambda d(z_x, y_x)+\lambda d(x, z_x)\geq u(z_x)+\lambda d(x, z_x)\geq \min_{z\in \ol{\O}} \{u(z)+ \lambda d(x, z)\}.
\]
Our proof of \eqref{dpp} is now complete. 

As mentioned before, \eqref{control eq} is just the McShane-Whitney Lipschitz extension. We can certainly obtain the Lipschitz regularity:
\beq\label{control eq4}
|u(x)-u(z)|\leq \lambda d(x, z) \quad \text{for any $x, z\in \Oba$.}
\eeq
In fact, the argument resulting in \eqref{control eq3} applies to all points $x, z\in\Oba$. Even if $x$ or $z$ appears on $\pO$, we can still get $y_z\in \pO$ satisfying \eqref{control eq2} and thus obtain \eqref{control eq3}. Interchanging the roles of $x$ and $z$ in $\Oba$, we get \eqref{control eq4} immediately. 

It is then clear that $|\nabla u|\leq \lambda$ in $\Omega$.  Let us now take any $x\in \Omega$ and any $r>0$ small such that $B_r(x)\subset \Omega$. Using \eqref{dpp} with $\O=B_r(x)$ we have
\[
\sup_{z\in B_r(x)} \big(u(x)-u(z) - \lambda d(x, z)\big)\geq 0,
\]
which implies that 
\[
\sup_{z\in B_r(x)\setminus \{x\}} \frac{u(x)-u(z)}{d(x, z)} \geq \lambda.
\]
Passing to the limit as $r\to 0$, we end up with $|\nabla^- u|(x)\geq \lambda$
for any $x\in \Omega$. Since $|\nabla^- u|\leq |\nabla u|$, we obtain $|\nabla^- u|=|\nabla u|= \lambda$ in $\Omega$, as desired. 
\end{proof}

The function $u$ defined by \eqref{control eq} obviously satisfies $u\leq g$ on $\pO$. Note that in general one cannot expect the $u= g$ holds on $\pO$ even in the Euclidean space. A simple example is as follows. Let $\X$ be the closed interval $[0, 1]$ equipped with the standard Euclidean metric and $\Omega=(0, 1)$. Assume that $0<\lambda <1$, $g(0)=0$ and $g(1)=1$. Then the function $u$ as in \eqref{control eq} can be directly computed: 
\[
u(x)=\lambda x, \quad x\in \Oba=\X=[0, 1].
\]
In particular, we have $u(1)= \lambda<g(1)$.  In general, one needs additional assumptions to guarantee the Dirichlet boundary condition in general metric spaces such as the $\lambda$-Lipschitz continuity of $g$ on $\pO$. See more details in \cite[Section 3.3]{LShZ}.

For our application in this work, we next present a comparison theorem for the Lipschitz Monge sub- and supersolutions in a proper geodesic space.  

\begin{thm}[Comparison principle for eikonal equation]\label{thm comparison monge}
Suppose that $(\X, d)$ is a proper geodesic space and $\Omega\subsetneq \X$ is a bounded domain. Let $u$ and $v$ be respectively a Monge subsolution and a Monge supersolution of \eqref{eikonal eq} with $\lambda >0$. Assume in addition that $u$ and $v$ are continuous in $\Oba$. 
If $u\leq v$ on $\pO$, then $u\le v$ in $\Oba$. 
\end{thm}

\begin{proof}
It is clear that $u$ and $v$ are bounded, since $u, v\in C(\Oba)$ and $\Omega$ is bounded. We therefore may assume that $u, v\geq 0$ by adding a positive constant to them. It suffices to show that $\mu u\le v$ in $\Omega$ for all $\mu\in (0,1)$. 

Assume by contradiction that 
there exists $\mu\in (0,1)$ such that $\sup_{\Omega}(\mu u-v)>0$. 
Due to the assumption that $u\leq v$ on $\pO$, we can find $x\in \Omega$ such that 
\beq\label{comparison monge1}
\sup_{\Omega}(\mu u-v)=\mu u(x)-v(x). 
\eeq
Since $v$ is a Monge supersolution of \eqref{eikonal eq}, there exists 
a sequence $\{y_n\}\subset \Omega$ converging to $x$ as $n\to \infty$ such that 
\[
\lim_{n\to \infty}\frac{[v(y_n)-v(x)]_-}{d(x, y_n)}\geq \lambda>0.
\]
Hence, by \eqref{comparison monge1} we have
\[
\begin{aligned}
\lambda \leq \limsup_{n\to \infty}\frac{\mu (u(x)-u(y_n))}{d(x, y_n)}
\leq \limsup_{y\to x}\frac{\mu (u(x)-u(y))}{d(x, y)}\leq \mu |\nabla^- u|(x).
\end{aligned}
\]
Noticing that $u$ is a Monge subsolution, we thus get $\lambda\leq \mu \lambda$, which is clearly a contradiction. 
\end{proof}

One can show similar comparison results in the case of general length spaces without assuming the spaces to be proper; see \cite[Theorem 4.2]{LShZ}. In this general case, the assumptions are slightly more complicated and the proof, involving Ekeland's variational principle 
 is more technical due to the possible lack of local compactness of the metric space.

\bibliographystyle{abbrv}%

\end{document}